\documentclass[11pt]{article}

\usepackage{amsmath,amsbsy,amsthm,amssymb}  
\usepackage{graphicx} 
\usepackage{booktabs} 
\usepackage{threeparttable} 
\usepackage{epstopdf}
\usepackage{hyperref}
\usepackage{url}      
\usepackage{subfigure} 
\usepackage{algorithm,algorithmic}

\DeclareGraphicsExtensions{.eps.gz,.eps,.epsi.gz,.epsi,.ps,.ps.gz}
\DeclareMathOperator*{\argmin}{\mathop{argmin}}
\DeclareMathOperator*{\argmax}{\mathop{argmax}}

\newtheorem{lemma}{Lemma}

\newtheorem{proposition}{Proposition}
\newtheorem{theorem}{Theorem}
\newtheorem{corollary}{Corollary}
\newtheorem{remark}{Remark}

\def\mathbold{\boldsymbol}

\def\bfh{\mathbold{h}}

\def\bfu{\mathbold{u}}

\def\bfI{\mathbold{I}}

\def\bfY{\mathbold{Y}}

\def\bfmu{\mathbold{\mu}}

\def\bfxi{\mathbold{\xi}}
\def\bfzeta{\mathbold{\zeta}}

\def\bfeta{\mathbold{\eta}}
\def\bflambda{\mathbold{\lambda}}
\def\bfpi{\mathbold{\pi}}

\def\bfpsi{\mathbold{\psi}}

\newcommand{\R}{{\mathbb{R}}}
\newcommand{\E}{{\mathbb{E}}}

\newcommand{\cH}{\mathcal{H}}

\newcommand{\cK}{\mathcal{K}}

\newcommand{\cN}{\mathcal{N}}

\newcommand{\cX}{\mathcal{X}}

\newcommand{\MSE}{\mathop{\mathrm{MSE}}}
\newcommand{\hMSE}{\widehat{\mathop{\mathrm{MSE}}}}
\newcommand{\epr}{\hfill\hbox{\hskip 4pt\vrule width 5pt
                  height 6pt depth 1.5pt}\vspace{0.5cm}\par}
\newcommand{\kf}{{\ff}}
\newcommand{\bff}{{\ff}}

\newcommand{\ee}{{\sf e}}

\newcommand{\ff}{{\sf f}}

\newcommand{\bflambdaproj}{\bflambda^{\textsc{proj}}}

\newcommand{\bflambdaq}{\bflambda^{Q}}
\newcommand{\lambdaq}{\lambda^{Q}}

\begin{document}
%
\title{Bayesian Model Averaging with Exponentiated Least Squares Loss}
%
%
%

\date{}
\author{
Dong Dai \\
IMS Health, PA, USA \\
\and
Lei Han \\
Tencent AI Lab, Shenzhen, China \\
\and
Ting Yang \\
Rutgers University, NJ, USA
\and  
Tong Zhang \\
Tencent AI Lab, Shenzhen, China \\
}

\maketitle

\begin{abstract}
The model averaging problem is to average multiple models to achieve a prediction accuracy not much worse than that of the best single model in terms of mean squared error. It is known that if the models are misspecified, model averaging is superior to model selection. Specifically, let $n$ be the sample size, then the worst case regret of the former decays at a rate of $O(1/n)$ while the worst case regret of the latter decays at a rate of $O(1/\sqrt{n})$. The recently proposed $Q$-aggregation algorithm \cite{DaiRigZhang12} solves the model averaging problem with the optimal regret of $O(1/n)$ both in expectation and in deviation; however it suffers from two limitations: (1) for continuous dictionary, the proposed greedy algorithm for solving $Q$-aggregation is not applicable; (2) the formulation of $Q$-aggregation appears ad hoc without clear intuition. This paper examines a different approach to model averaging by considering a Bayes estimator for deviation optimal model averaging by using exponentiated least squares loss. We establish a primal-dual relationship of this estimator and that of $Q$-aggregation and propose new algorithms that satisfactorily resolve the above mentioned limitations of $Q$-aggregation.
\end{abstract}

\section{Introduction}

This paper considers the model averaging problem, where the goal is to average multiple models in order to achieve improved prediction accuracy.
Let $x_1, \ldots, x_n$ be $n$ given design points from a space $\cX$, let $\cH=\{f_1, \ldots, f_M\}$ be a given dictionary of real valued functions on $\cX$ and denote $\bff_j=(f_j(x_1),\ldots,f_j(x_n))^\top \in \R^n$ for each $j$. The goal is to estimate an unknown regression function $\eta: \cX \to \R$ at the design points based on observations $$
y_i = \eta(x_i) +\xi_i\,,
$$
where $\xi_1, \ldots, \xi_n$ are i.i.d. variables from $\cN(0, \sigma^2)$.

The performance of an estimator $\hat \eta$ is measured by its mean squared error (MSE) defined by
$$
\MSE(\hat \eta)=\frac{1}{n} \sum_{i=1}^n (\hat \eta(x_i)-\eta(x_i))^2\,.
$$
We want to find an estimator $\hat \eta$ that mimics the function in the dictionary with the smallest MSE.
Formally, a good estimator $\hat \eta$ should satisfy the following \emph{exact oracle inequality} in a certain probabilistic sense:
\begin{equation}
\label{EQ:oiDelta}
\MSE(\hat \eta) \le \min_{j=1, \dots, M} \MSE(f_j) + \Delta(n,M,\sigma^2)\,,
\end{equation}
where the remainder term $\Delta>0$ should be as small as possible.

The problem of model averaging has been well-studied, and it is known (see, e.g., \cite{Tsy03,Rig12}) that the smallest possible order
for $\Delta(n,M,\sigma^2)$ is $\sigma^2\log M/n$ for oracle inequalities in expectation, where ``the smallest possible'' is understood in the following minimax sense. There exists a dictionary $\cH=\{f_1, \ldots, f_M\}$ such that the following lower bound holds. For any estimator $\hat \eta$, there exists a regression function $\eta$ such that
$$
\E\MSE(\hat \eta) \ge \min_{j=1, \dots, M} \MSE(f_j) + C\sigma^2 \frac{\log M}{n}\,
$$
for some positive constant $C$. It also implies that the lower bound holds not only in expectation but also with positive probability.

Although our goal is to achieve an MSE as close as that of the best model in $\cH$,
it is known (see Theorem~2.1 of \cite{RigTsy12}) that there
exists a dictionary~$\cH$ such that any estimator $\hat \eta$ taking values restricted to the elements of $\cH$  (such an estimator is  referred to as a {\em model selection estimator}) cannot
achieve an oracle inequality of form~\eqref{EQ:oiDelta} with a remainder term of order smaller than $\sigma\sqrt{(\log M)/n}$;
in other words, model selection is suboptimal for the purpose of competing with the best single model from a given family.

Instead of {\em model selection}, we can employ {\em model averaging} to derive oracle inequalities of form \eqref{EQ:oiDelta} that achieves the optimal regret
in expectation (see the references in \cite{RigTsy12}). More recently, several work has produced optimal oracle inequalities for model averaging
that not only hold in expectation but also in deviation~\cite{Aud08, LecMen09, GaiLec11, DaiZha11, Rig12, DaiRigZhang12}.
In particular, the current work is closely related to the $Q$-aggregation estimator investigated in \cite{DaiRigZhang12} which solves the optimal model averaging problem both in expectation and in deviation with a remainder term $\Delta(n,M,\sigma^2)$ of order $O(1/n)$; the authors also proposed a greedy algorithm GMA-0 for $Q$-aggregation which improves the Greedy Model Averaging (GMA) algorithm firstly proposed by \cite{DaiZha11}. Yet there are still two limitations of $Q$-aggregation:
(1) $Q$-aggregation can be generalized for continuous candidates dictionary $\cH$, but the greedy model averaging method GMA-0 can not be applied in such setting;
(2) $Q$-aggregation can be regarded intuitively as regression with variance penalty, but it lacks a good theoretical interpretation.

In this paper we introduce a novel method called {\em Bayesian Model Averaging with Exponentiated Least Squares Loss} (BMAX).
We note that the previously studied exponential weighted model aggregation estimator EWMA (e.g., \cite{RigTsy12}) is the Bayes estimator under the least squares loss (posterior mean), which leads to optimal regret in expectation but is suboptimal in deviation. In contrast, the new BMAX model averaging estimator is essentially a Bayes estimator under an appropriately defined {\em exponentiated least squares loss}, and we will show that the $Q$-aggregation formulation (with Kullback-Leibler entropy) in \cite{DaiRigZhang12} is essentially a dual representation of the newly introduced BMAX formulation, and it directly implies the optimality of the aggregate by BMAX. Computationally, the new model aggregation method BMAX can be approximately solved by a greedy model averaging algorithm and a gradient descend algorithm which is applicable to continuous dictionary. In summary, this paper establishes a natural Bayesian interpretation of $Q$-aggregation, and provides additional computational procedures that are applicable for the continuous dictionary setting. This relationship provides deeper understanding for model averaging procedures, and resolves the above mentioned limitations of the $Q$-aggregation scheme.

\section{Notations}

This section introduces some notations used in this paper. In the following, we denote by $\bfY=(y_1,\ldots,y_n)^\top$ the observation vector, $\bfeta=(\eta(x_1),\ldots,\eta(x_n))^\top$ the model output, and $\bfxi=(\xi_1,\ldots,\xi_n)^\top$ the noise vector. The underlying statistical model can be expressed as
\begin{equation}
\bfY= \bfeta+\bfxi \;,
\label{EQ:gauss}
\end{equation}
with $\bfxi\sim N(0,\sigma^2\bfI_n)$.
We also denote $\ell_2$ norm as $\|\bfY\|_2=(\sum_{i=1}^n y_i^2)^{1/2}$, and the inner product as $\langle \bfxi, \bff \rangle = \bfxi^\top \bff$.
Let $\Lambda^M$ be the flat simplex in $\R^M$ defined by
$$
\Lambda^M = \left\{ \bflambda=(\lambda_1,\ldots,\lambda_M)^\top \in \R^M : \lambda_j \geq 0, \sum_{j=1}^M \lambda_j= 1\right\} \;,
$$
and $\bfpi=(\pi_1, \ldots, \pi_M)^\top \in \Lambda^M$ be a given prior.

Each $\bflambda \in \Lambda^M$ yields a model averaging estimator as
$f_{\bflambda} = \sum_{j=1}^M \lambda_j f_j $; that is,
using the vector notation $\kf_{\bflambda}=(f_{\bflambda}(x_1),\ldots, f_{\bflambda}(x_n))^\top$ we have
$\kf_{\bflambda} =  \sum_{j=1}^M \lambda_j \bff_j$.
The Kullback-Leibler divergence for $\bflambda,\bfpi \in \Lambda^M$ is defined as
\begin{eqnarray*}
\cK(\bflambda,\bfpi)= \sum_{j=1}^M \lambda_j \log(\lambda_j/\pi_j) \;,
\end{eqnarray*}
and in the definition we use the convention $0\cdot\log(0)=0$.
For matrices $A, B \in \R^{n\times n}$, $A\geq B$ indicates that $A-B$ is positive semi-definite.

\section{Bayesian Model Averaging with Exponentiated Least Squares Loss}
\label{sec:bmax}
The traditional Bayesian model averaging estimator is the exponential weighted model averaging estimator EWMA \cite{RigTsy12} which optimizes
the least squares loss. Although the estimator is optimal in expectation, it is suboptimal in deviation \cite{DaiRigZhang12}.
In this section we introduce a different Bayesian model averaging estimator called BMAX that optimizes an exponentiated least squares loss.

In order to introduce the BMAX estimator, we consider the following Bayesian framework, where we should be noted that
the assumptions below are only used to derive BMAX, and these assumptions are not used in our theoretical analysis.
$\bfY$ is a normally distributed observation vector
with mean $\bfmu=(\mu_1,\ldots,\mu_M)^\top$ and covariance matrix $\omega^2\bfI_n$:
\begin{equation}
\bfY|\bfmu \sim N(\bfmu,\omega^2\bfI_n) \;,
\end{equation}
and for $j=1,\ldots,M$, the prior for each model $\bff_j$ is
\begin{equation}
\pi(\bfmu=\bff_j) = \pi_j \;.
\end{equation}

In this setting, the posterior distribution of $\bfmu$ given $\bfY$ is
\begin{align*}
p(\bfmu=\bff_j|\bfY)
= \frac{p(\bfY|\bfmu=\bff_j)p(\bfmu=\bff_j)}{\sum_{j=1}^M p(\bfY|\bfmu=\bff_j)p(\bfmu=\bff_j)}
= \frac{\exp\left(-\frac{\|\bff_j-\bfY\|_2^2}{2\omega^2}\right) \pi_j }{\sum_{j=1}^M \exp\left(-\frac{\|\bff_j-\bfY\|_2^2}{2\omega^2}\right) \pi_j } .
\end{align*}
In the Bayesian decision theoretical framework considered in this paper,
the quantity of interest is $\bfeta=\E\bfY$, and we consider a  loss function $L(\bfpsi,\bfmu)$ which we would like to minimize with respect
to the posterior distribution.
The corresponding Bayes estimator $\hat\bfpsi$ minimizes the posterior expected loss from $\bfmu$ as follows:
\begin{equation}
\hat\bfpsi = \argmin_{\bfpsi\in \R^n} \E \left[ L(\bfpsi,\bfmu)|\bfY \right] \;.
\label{EQ:def-Bayes}
\end{equation}
It is worth pointing out that the above Bayesian framework is only used to obtain decision theoretically motivated model averaging estimators (Bayesian estimators have good theoretical properties such as admissibility, etc).
In particular we do not assume that the model itself is correctly specified. That is, in this paper we allow misspecified models, where
the parameters $\bfmu$ and $\omega^2$ are not necessarily equal to the true mean $\bfeta$ and the true variance $\sigma^2$ in \eqref{EQ:gauss}, and $\bfeta$ does not necessarily belong to the dictionary $\{\bff_1,\ldots,\bff_M\}$.

The Bayesian estimator of \eqref{EQ:def-Bayes} depends on the underlying loss function $L(\cdot,\cdot)$.
For example, under the standard least squares loss $L(\bfpsi,\bfmu)=\|\bfpsi-\bfmu\|_2^2$, the Bayes estimator is the posterior mean, which leads to
the Exponential Weighted Model Aggregation (EWMA) estimator \cite{RigTsy12} :
\begin{equation}
\bfpsi_{\ell_2}(\omega^2) = \frac{ \sum_{j=1}^M \exp\left(-\frac{\|\bff_j-\bfY\|_2^2}{2\omega^2}\right) \pi_j \bff_j }{\sum_{j=1}^M \exp\left(-\frac{\|\bff_j-\bfY\|_2^2}{2\omega^2}\right) \pi_j } .
\label{EQ:defpsil2}
\end{equation}
This estimator is optimal in expectation ~\cite{DalTsy07, DalTsy08}, but suboptimal in deviation ~\cite{DaiRigZhang12}.

In this paper, we introduce the following {\em exponentiated least squares loss} motivated from the exponential moment technique for proving large
deviation tail bounds for sums of random variables:
\begin{equation}
L(\bfpsi,\bfmu)=\exp\left(\frac{1-\nu}{2\omega^2}\|\bfpsi-\bfmu\|_2^2\right) \;,
\label{EQ:def-exploss}
\end{equation}
where the parameter $\nu\in(0,1)$.
It is easy to verify that the Bayes estimator defined by~\eqref{EQ:def-Bayes} with the loss function defined in \eqref{EQ:def-exploss} can be written as
\begin{equation}
\bfpsi_{X}(\omega^2,\nu)  = \argmin_{\bfpsi\in\R^n} J(\bfpsi) \;,
\label{EQ:defpsiex}
\end{equation}
where
\begin{equation}
J(\bfpsi) = \sum_{j=1}^M \pi_j \exp\left(-\frac{1}{2\omega^2}\|\bff_j-\bfY\|_2^2+\frac{1-\nu}{2\omega^2}\|\bfpsi-\bff_j\|_2^2\right) \;.
\label{EQ:defJ}
\end{equation}
The estimator $\bfpsi_{X}(\omega^2,\nu)$ will be referred to as the Bayesian model aggregation estimator with exponentiated least squares loss (BMAX).

To minimize $J(\bfpsi)$, it is equivalent to minimize $\log J(\bfpsi)$. Lemma~\ref{LEM:logJ-2ord-convexity} below shows the {\em strong convexity} and {\em smoothness} (under some conditions) of $\log J(\bfpsi)$.

Given $\nu\in(0,1)$ and $\omega>0$, we define
\begin{align}
A_1 =& \frac{1-\nu}{\omega^2} \;; \label{EQ:defA1}
\end{align}
moreover, if the $\ell_2$-norm of every $\bff_j$ is bounded by a constant $L\in\R$:
\begin{equation}
\|\bff_j\|_2 \leq L  \;,\;\; \forall\;j=1,\ldots,M \;,
\label{CON:f-l2normbound}
\end{equation}
we define $A_2$ and $A_3$ as
\begin{align}
A_2 =& \frac{1-\nu}{\omega^2}+ \left(\frac{1-\nu}{\omega^2}\right)^2 L^2 \;, \label{EQ:defA2} \\
A_3 =& \left(\frac{1-\nu}{\omega^2}\right) L^2 +  \left(\frac{1-\nu}{\omega^2}\right)^2 L^4 \;.  \label{EQ:defA3}
\end{align}

\begin{lemma}
For any $\bfpsi\in\R^n$, define the Hessian matrix of  $\log J(\bfpsi)$ as
$\nabla^2 \log J(\bfpsi)= \frac{\partial^2 \log J(\bfpsi)}{\partial{\bfpsi}\partial{\bfpsi^\top}}$,
then we have
\begin{equation}
\nabla^2 \log J(\bfpsi) \geq A_1 \bfI_n \;.
\label{INEQ:logJ-Hes-LB}
\end{equation}
If $\{\bff_1,\ldots,\bff_M\}$ satisfies condition~\eqref{CON:f-l2normbound}, then
\begin{equation}
\nabla^2 \log J(\bfpsi) \leq A_2 \bfI_n \;,
\label{INEQ:logJ-Hes-UB}
\end{equation}
where $A_1$ and $A_2$ are defined in \eqref{EQ:defA1} and \eqref{EQ:defA2}.
\label{LEM:logJ-2ord-convexity}
\end{lemma}

\section{Dual Representation and $Q$-aggregation}
\label{sec:duality}

In this section, we will show that the \emph{$Q$-aggregation} scheme of~\cite{DaiRigZhang12} with the standard Kullback-Leibler entropy
solves a dual representation of the BMAX formulation defined by~\eqref{EQ:defpsiex}.

Given $\bfY$ and $\{\bff_1,\ldots,\bff_M\}$, $Q$-aggregation $\kf_{\bflambdaq}$ is defined as follows:
\begin{equation}
\kf_{\bflambdaq}= \sum_{j=1}^M \lambdaq_j \bff_j \;,
\label{EQ:defQagg}
\end{equation}
where $\bflambdaq = (\lambdaq_1,\ldots,\lambdaq_M)^\top \in \Lambda^M$ such that
\begin{equation}
\label{EQ:deflq}
\bflambdaq \in \argmin_{\bflambda \in \Lambda^M} Q(\bflambda) \;,
\end{equation}
\begin{equation}
\label{EQ:defQ}
Q(\bflambda)=  \|\kf_{\bflambda}-\bfY\|_2^2 + \nu \sum_{j=1}^M \lambda_j \|\bff_j-\kf_{\bflambda} \|_2^2 + 2\omega^2 \cK_\rho(\bflambda,\bfpi) \;,
\end{equation}
for some $\nu\in (0,1)$, where the $\rho$-entropy $\cK_\rho(\bflambda,\bfpi)$ is defined as
\begin{equation}
\cK_\rho(\bflambda, \bfpi)=\sum_{j=1}^M \lambda_j \log \left(\frac{\rho(\lambda_j)}{\pi_j}\right) \;,
\label{EQ:defKL-rho}
\end{equation}
where $\rho$ is a real valued function on $[0,1]$ satisfying
\begin{align}
\rho(t)\ge t \;,  \quad
t\log\rho(t) \text{ is convex} \;.
\label{CON:rho}
\end{align}

When $\rho(t)=t$, $\cK_\rho(\bflambda, \bfpi)$ becomes $\cK(\bflambda,\bfpi)$, i.e., the Kullback-Leibler entropy. When $\rho(t)=1$, $\cK_\rho(\bflambda, \bfpi)=\sum_{j=1}^M \lambda_j \log(1/\pi_j)$, a linear entropy in $\Lambda^M$, and in particular the penalty $\cK_\rho(\bflambda, \bfpi)$ in ~\eqref{EQ:defQ} becomes a constant when $\bfpi$ is a flat prior.

Now, to illustrate duality, we shall first introduce a function $T:\R^n\rightarrow\R$ as
\begin{align}
T(\bfh) = &-\frac{\nu}{1-\nu}\|\bfh-\bfY\|_2^2
-2\omega^2 \log\left(\sum_{j=1}^M \pi_j \exp\left(-\frac{\nu}{2\omega^2}\|\bff_j-\bfh\|_2^2\right)\right) \;,
\label{EQ:defT}
\end{align}
and denote the maximizer of $T(\bfh)$ as
\begin{equation}
\hat{\bfh} = \argmax_{\bfh \in \R^n} T(\bfh) \;.
\label{EQ:defhat-h-T}
\end{equation}

Define function $S: \Lambda^M \times \R^n \rightarrow \R $ as
\begin{equation}
S(\bflambda,\bfh) = -\frac{\nu}{1-\nu}\|\bfh-\bfY\|_2^2 + \nu \sum_{j=1}^M \lambda_j\|\bff_j-\bfh\|_2^2+ 2\omega^2 \cK(\bflambda,\bfpi) \;.
\label{def:S}
\end{equation}
It is not difficult to verify that for $\nu \in (0,1)$,
$S(\bflambda,\bfh)$ is convex in $\bflambda$ and concave in $\bfh$. 
The following duality lemma states the relationship between $\hat\bfh$ and $\kf_{\bflambda^Q}$.
\begin{lemma}
When $\rho(t)=t$, we have the following result
\[
Q(\bflambda) = \max_{\bfh \in \R^n} S(\bflambda,\bfh) , \qquad 
T(\bfh) = \min_{\bflambda \in \Lambda^M} S(\bflambda,\bfh) .
\]
\begin{align*}
\min_{\bflambda \in \Lambda^M} Q(\bflambda) = \min_{\bflambda \in \Lambda^M} \max_{\bfh \in \R^n} S(\bflambda,\bfh)
=\max_{\bfh \in \R^n} \min_{\bflambda \in \Lambda^M} S(\bflambda,\bfh)
=\max_{\bfh \in \R^n} T(\bfh) ,
\end{align*}
where the equality is achieved at $(\bflambdaq, \hat\bfh)$.
Moreover, we have
\[
\left\{(\bflambdaq, \hat\bfh)\right\} =  A\cap B ,
\]
where $A$ and $B$ are two hyper-surfaces  in $\Lambda^M \times \R^n$
defined as
\begin{align}
A=&\left\{(\bflambda,\bfh)\in \Lambda^M \times \R^n: \bfh=\frac{1}{\nu}\bfY-\frac{1-\nu}{\nu}\kf_{\bflambda}\right\} \;, 
\nonumber
\\
B=&\Bigg\{(\bflambda,\bfh)\in \Lambda^M \times \R^n:
\lambda_j =
  \frac{\exp\left(-\frac{\nu}{2\omega^2}\|\bff_j-\bfh\|_2^2\right)\pi_j}{\sum_{i=1}^M
    \exp\left(-\frac{\nu}{2\omega^2}\|\bff_i-\bfh\|_2^2\right)\pi_i}\Bigg\}
\; .
\label{def:hyppl}
\end{align}
\label{LEM:Q-S-T}
\end{lemma}
Lemma~\ref{LEM:Q-S-T} states that, $(\bflambdaq, \hat\bfh)$ is the only joint of hyper-surfaces  $A$ and $B$, and the only saddle point of function $S(\bflambda,\bfh)$ over space $\Lambda^M \times \R^n$.

With $T(\bfh)$ defined as in \eqref{EQ:defT}, we can employ the transformation
$
\bfh= \frac{1}{\nu}\bfY-\frac{1-\nu}{\nu}\bfpsi ,
$
and it is easy to verify that
\begin{align}
T(\bfh)
=  -2\omega^2 \log\left(J(\bfpsi)\right) \;,
\label{EQ:T-J}
\end{align}
where $J(\bfpsi)$ is defined in \eqref{EQ:defJ}. Since $J(\bfpsi)$ is strictly convex, $T(\bfh)$ is strictly concave and $\hat\bfh$ is unique.
It follows that maximizing $T(\bfh)$ is equivalent to minimizing $J(\bfpsi)$, and thus
\[
\hat\bfh= \frac{1}{\nu}\bfY-\frac{1-\nu}{\nu}\bfpsi_{X}(\omega^2,\nu) \;.
\]
We can combine this representation with
\[
\hat\bfh = \frac{1}{\nu}\bfY-\frac{1-\nu}{\nu}\kf_{\bflambdaq} \;
\]
from Lemma~\ref{LEM:Q-S-T} to obtain $\bfpsi_{X}(\omega^2,\nu)= \kf_{\bflambdaq}$.
Therefore, we directly have the following relationship.
\begin{theorem}\label{TH:psiex-Q}
When $\rho(t)=t$,
\[
\bfpsi_{X}(\omega^2,\nu) = \kf_{\bflambdaq} \;,
\]\noindent
where $\bfpsi_{X}(\omega^2,\nu)$ is defined by \eqref{EQ:defpsiex} and \eqref{EQ:defJ}, and $\kf_{\bflambdaq}$ is defined by \eqref{EQ:defQagg},\eqref{EQ:deflq} and \eqref{EQ:defQ}.
\end{theorem}

Theorem~\ref{TH:psiex-Q} states that, when $\rho(t)=t$, $\cK_\rho(\bflambda,\bfpi)$ becomes the Kullback-Leibler entropy, and $Q$-aggregation with the Kullback-Leibler entropy leads to an estimator $\kf_{\bflambdaq}$ that is essentially a dual representation of the BMAX estimator $\bfpsi_{X}(\omega^2,\nu)$. It follows that, $\bfpsi_{X}(\omega^2,\nu)$ should share the same optimality (both in expectation and deviation) as $\kf_{\bflambdaq}$ in solving the model averaging problem (optimality of $\kf_{\bflambdaq}$ is shown in Theorem~3.1 of \cite{DaiRigZhang12} with more general $\cK_\rho(\bflambda,\bfpi)$, where $\rho(t)$ only needs to satisfy the condition~\eqref{CON:rho}).

However, unlike the primal objective function $J(\bfpsi)$ which is defined on $\R^n$, the dual objective function $Q(\bflambda)$ is defined on $\R^M$.
When $M$ is large or infinity, the optimization of $Q(\bflambda)$ is non-trivial.
Although greedy algorithms are proposed in \cite{DaiRigZhang12}, they cannot handle the standard KL-divergence; instead, they can only work with the
linear entropy where $\rho(t)=1$; it gives a larger penalty than the standard KL-divergence (and thus worse resulting oracle inequality),
and it cannot be generalized to handle continuous dictionaries
(because in such case the linear entropy with $\rho(t)=1$ will always be $+\infty$). Therefore, the numerical greedy procedures of \cite{DaiRigZhang12}
converge to a solution with a worse oracle bound than that of the solution for the primal formulation considered in this paper.

The following two corollaries (Corollary~\ref{COR:opt-psiex0} and Corollary~\ref{COR:opt-psiex}) are listed for illustration convenience. They are directly derived from Theorem~\ref{TH:psiex-Q} and the optimality of $\kf_{\bflambdaq}$ in Theorem~3.1 of \cite{DaiRigZhang12}, so we omit their proofs.

\begin{corollary}
Assume that $\nu \in (0,1)$ and $\omega^2 \geq \frac{\sigma^2}{\min(\nu,1-\nu)}$.
For any $\bflambda \in \Lambda^M$, the following oracle inequality holds
\begin{align}
\|\bfpsi_{X}(\omega^2,\nu)-\bfeta\|_2^2 \leq 
\nu \sum_{j=1}^M \lambda_j \|\bff_j-\bfeta\|_2^2 
+(1-\nu)\left\|\kf_{\bflambda}- \bfeta\right\|_2^2
+2\omega^2 \cK(\bflambda,\bfpi\delta)\;,
\label{INEQ:psiex-dev}
\end{align}
with probability at least $1-\delta$. Moreover,
\begin{align}
\E \|\bfpsi_{X}(\omega^2,\nu)-\bfeta\|_2^2
\leq &\ \nu \sum_{j=1}^M \lambda_j \|\bff_j-\bfeta\|_2^2+ (1-\nu)\left\|\kf_{\bflambda}- \bfeta\right\|_2^2
+2\omega^2 \cK(\bflambda,\bfpi) \;
\nonumber
\\
\leq &\left\|\kf_{\bflambda}- \bfeta\right\|_2^2 +
\nu \sum_{j=1}^M \lambda_j \|\bff_j-\kf_{\bflambda}\|_2^2
+2\omega^2 \cK(\bflambda,\bfpi)\;.
\label{INEQ:psiex-exp}
\end{align}
\label{COR:opt-psiex0}
\end{corollary}
Corollary \ref{COR:opt-psiex0} implies that when the term $\nu \sum_{j=1}^M \lambda_j \|\bff_j-\kf_{\bflambda}\|_2^2$ and the divergence term $K(\bflambda,\bfpi)$ are small, $\bfpsi_{X}(\omega^2,\nu)$ can compete with an arbitrary $\kf_{\bflambda}$ in the convex hull with any $\bflambda \in \Lambda^M$. Actually, we can obtain an oracle inequality that competes with the best single model, which is the situation that $\bflambda$ is at a vertex of the simplex $\Lambda^M$:

\begin{corollary}
Under the assumptions of Corollary~\ref{COR:opt-psiex0}, $\bfpsi_{X}(\omega^2,\nu)$ satisfies
\begin{equation}
\|\bfpsi_{X}(\omega^2,\nu)-\bfeta\|_2^2 \leq \min_{j\in{1,\ldots,M}}\left\{\|\bff_j-\bfeta\|_2^2 + 2\omega^2 \log\left(\frac{1}{\pi_j\delta}\right)\right\} \;,
\label{INEQ:psiex-one-dev}
\end{equation}
with probability at least $1-\delta$. Moreover,
\begin{equation}
\E \|\bfpsi_{X}(\omega^2,\nu)-\bfeta\|_2^2 \leq \min_{j\in{1,\ldots,M}}\left\{\|\bff_j-\bfeta\|_2^2 + 2\omega^2 \log\left(\frac{1}{\pi_j}\right)\right\} \;.
\label{INEQ:psiex-one-exp}
\end{equation}
\label{COR:opt-psiex}
\end{corollary}

It is also worth pointing out that the condition $\omega^2 \geq \frac{\sigma^2}{\min(\nu,1-\nu)}$ implies that $\omega^2$ is at least greater than $2\sigma^2$ (when $\nu=1/2$),
and intuitively this inflation of noise allows the Bayes estimator to handle misspecification of the true mean $\bfeta$, which is not necessarily included in the dictionary $\cH$. Similar observations were found by~\cite{leung2006information}.

Finally we note that in the Bayesian framework stated in this section, when we change the underlying loss function $L(\bfpsi,\bfmu)$ from the standard least squares loss to the exponentiated least squares loss \eqref{EQ:def-exploss}, Bayes estimator changes from EWMA which is optimal only in expectation to BMAX which is optimal both in expectation and in deviation. The difference is that the least squares loss only controls the bias, while the exponentiated least squares loss controls both bias and variance (as well as higher order moments) simultaneously. This can be seen by using Taylor expansion
\begin{align*}
\exp\left(\frac{1-\nu}{2\omega^2}\|\bfpsi-\bfmu\|_2^2\right)
= 1+ \frac{1-\nu}{2\omega^2}\|\bfpsi-\bfmu\|_2^2
+ (1/2)\left(\frac{1-\nu}{2\omega^2}\|\bfpsi-\bfmu\|_2^2\right)^2 + \cdots \;.
\end{align*}
Since deviation bounds require us to control high order moments, the exponentiated least squares loss is naturally suited for obtaining deviation bounds.

It is also natural to extend Corollary~\ref{COR:opt-psiex0} from discrete candidates dictionary $\cH=\{\bff_1,\ldots,\bff_M\}$ to infinite dictionary ($M=\infty$) as well as continuous dictionary. For example, given a matrix $X \in \R^{n \times d}$, we may consider a continuous dictionary parameterized by vector $w$ as $\cH_{\Omega}=\{\bff_w: \bff_w= X w\in\R^n \}$ where $ \Omega=\{w: w \in \R^d\}$. Then, we have the following result.

\begin{corollary}
Assume $\nu \in (0,1)$ and for any distribution $p(w)$ over the parameter space $\Omega$, the divergence term $\cK(\mathbold{p},\mathbold{\pi})$ is finite for given prior $\pi(w)$. Then, if $\omega^2 \geq \frac{\sigma^2}{\min(\nu,1-\nu)}$, we have the oracle inequality
\begin{align}
\|\hat{\bfpsi}-\bfeta\|_2^2 &\leq \nu \int_\Omega \|\bff_w-\bfeta\|_2^2 p(w)\;dw
+(1-\nu)\left\|\int_\Omega\bff_wp(w)\;dw-\bfeta\right\|_2^2 + 2\omega^2 \cK(\mathbold{p},\bfpi\delta)
\label{INEQ:psiex-dev-con}
\end{align}
with probability at least $1-\delta$. Moreover,
\begin{align}
\E \|\hat{\bfpsi}-\bfeta\|_2^2 
\leq &\ \nu \int_\Omega \|\bff_w-\bfeta\|_2^2 p(w)\;dw 
+(1-\nu)\left\|\int_\Omega\bff_wp(w)\;dw -\bfeta\right\|_2^2
+2\omega^2 \cK(\mathbold{p},\bfpi) \;
\nonumber
\\
\leq &\left\|\int_\Omega\bff_wp(w)\;dw -\bfeta\right\|_2^2 +
\nu \int_\Omega \left\|\bff_w-\int_\Omega\bff_wp(w)\;dw\right\|_2^2 p(w)\;dw
+2\omega^2 \cK(\mathbold{p},\bfpi) \;,
\label{INEQ:psiex-exp-con}
\end{align}
where
\begin{align*}
\hat{\bfpsi} = \argmin_{\bfpsi\in\R^n} \int_\Omega \exp\Bigg(&-\frac{1}{2\omega^2}\|\bff_w-\bfY\|_2^2
+\frac{1-\nu}{2\omega^2}\|\bff_w-\bfpsi\|_2^2\Bigg) \pi(w)\;dw.
\end{align*}
\label{COR:opt-psiex-con}
\end{corollary}

When the distribution $\mathbold{p}$ on $w \in \R^d$ is concentrated around a single model, for example, $\mathbold{p}$ is an uniform distribution on a small ball $\{w: \|w-w_0\|_2 \leq r\}$ for some small $r>0$, we have $\int_\Omega\bff_wp(w)\;dw=\kf_{w_0}$ and the term $\frac{1}{n}\int_\Omega  \|\bff_w-\kf_{w_0}\|_2^2 p(w) dw\leq \frac{r^2}{n}\|X\|_F^2$ is small if $\frac{1}{n}\|X\|_F^2$ is bounded ($\|\cdot\|_F$ is the matrix Frobenius norm). Although a small $r$ would lead to a larger divergence term $\cK(\mathbold{p},\bfpi)$ and there is a tradeoff between $\int_\Omega  \|\bff_w-\kf_{w_0}\|_2^2 p(w) dw$ and $\cK(\mathbold{p},\bfpi)$, by dividing the number of observations $n$ on both sides of Eq.~(\ref{INEQ:psiex-exp-con}), the term $\frac{1}{n}\cK(\mathbold{p},\bfpi)$ could also be small as long as $n$ is sufficiently large. Therefore, Corollary~\ref{COR:opt-psiex-con} implies that  $\hat{\bfpsi}$ can compete with a single model $\kf_{w_0}$, similar to the case of discrete dictionary.

\section{Algorithms to solve BMAX}

In this section, we propose two algorithms, the Greedy Model Averaging (GMA-BMAX) algorithm and the Gradient Descent (GD-BMAX) algorithm, to solve BMAX. The convergence rates of both algorithms will be shown. Specifically, denote $k$ as the number of iterations in the algorithms, GMA-BMAX algorithm has a converge rate of $O(1/k)$ and GD-BMAX algorithm converges with a geometric rate of $O(q^k)$ for some $q\in(0,1)$. Oracle inequalities will be derived for the estimators in the iterations for both the algorithms.

Strong convexity of $\log J(\bfpsi)$ directly implies that the minimizer $\bfpsi_{X}(\omega^2,\nu)$ is unique. Moreover, it implies the following proposition which shows that an estimator that approximately minimizes $\log J(\bfpsi)$ satisfies an oracle inequality slightly worse than that of $\bfpsi_{X}(\omega^2,\nu)$ in Corollary~\ref{COR:opt-psiex0}. This result suggests that we can employ appropriate numerical procedures to approximately solve \eqref{EQ:defpsiex}, and Corollary~\ref{COR:opt-psiex0} implies an oracle inequality for such approximate solutions.
\begin{proposition}\label{prop:approx-sol}
Let $\hat{\bfpsi}$ be an $\epsilon$-approximate minimizer of $\log J(\bfpsi)$  for some $\epsilon>0$ that $\log J(\hat{\bfpsi}) \leq   \min_{\bfpsi} \log J(\bfpsi) + \epsilon$. Then, we have
\[
\|\hat{\bfpsi}-\bfeta\|_2^2 \leq \|\bfpsi_{X}(\omega^2,\nu)-\bfeta\|_2^2 + 2 \sqrt{2\epsilon/A_1} \|\bfpsi_{X}(\omega^2,\nu)-\bfeta\|_2 + \frac{2 \epsilon}{A_1} .
\]
\end{proposition}
Next, we present the numerical algorithms to solve the BMAX problem.

\subsection{Greedy Model Averaging Algorithm (GMA-BMAX)}

The GMA-BMAX algorithm given in Algorithm~\ref{ALG:GMA-BMAX} is a greedy algorithm that adds at most one function from the dictionary $\cH$ at each iteration. This feature is attractive as it outputs a $k$-sparse solution that depends on at most $k$ functions from the dictionary after $k$ iterations.
Similar algorithms for model averaging have appeared in \cite{DaiZha11} and \cite{DaiRigZhang12}.

\begin{algorithm}[t]
\caption{Greedy Model Averaging Algorithm (GMA-BMAX)}
\label{ALG:GMA-BMAX}
\begin{algorithmic}
\REQUIRE Noisy observation $\bfY$, dictionary $\cH=\{f_1,\ldots,f_M\}$, prior $\bfpi \in \Lambda^M$, parameters $\nu, \omega$.
\ENSURE  Aggregate estimator  $\bfpsi^{(k)}$.
\STATE  Let $\bfpsi^{(0)}=0$;
\smallskip
\FOR{$k=1, 2, \ldots$}
\STATE Set $\alpha_k=\frac{2}{k+1}$;
\STATE  $J^{(k)}=\argmin_j \log J(\bfpsi^{(k-1)} + \alpha_k(\bff_j - \bfpsi^{(k-1)}))$;
\STATE  $\bfpsi^{(k)}= \bfpsi^{(k-1)} + \alpha_k(\bff_{J^{(k)}} - \bfpsi^{(k-1)})$;
\ENDFOR
\end{algorithmic}
\end{algorithm}

The following proposition follows from the standard analysis in \cite{FraWol56,Jon92,Barron93}.
It shows that the estimator $\bfpsi^{(k)}$ from Algorithm~\ref{ALG:GMA-BMAX} converges to $\bfpsi_{X}(\omega^2,\nu)$.

\begin{proposition}
For $\bfpsi^{(k)}$as defined in Algorithm~\ref{ALG:GMA-BMAX} (GMA-BMAX), if $\{\bff_1,\ldots,\bff_M\}$ satisfies condition~\eqref{CON:f-l2normbound}, then
\begin{equation}
\log J (\bfpsi^{(k)}) \leq \log J(\bfpsi_{X}(\omega^2,\nu)) + \frac{8A_3}{k+3} \;.
\label{INEQ:GMA-approxlogJ}
\end{equation}
\label{PROP:GMA-approxlogJ}
\end{proposition}
Proposition~\ref{PROP:GMA-approxlogJ} states that, after running the GMA-BMAX algorithm for $k$ steps to obtain $\bfpsi^{(k)}$, the corresponding objective value $\log J (\bfpsi^{(k)})$ converges to the optimal objective value $\log J(\bfpsi_{X}(\omega^2,\nu))$ at a rate $O(1/k)$.
Combining this result with Proposition~\ref{prop:approx-sol}, we obtain the following
oracle inequality, which shows that the regret of the estimator $\bfpsi^{(k)}$ after running $k$ steps of GMA-BMAX converges to
that of $\bfpsi_{X}(\omega^2,\nu)$ in Corollary~\ref{COR:opt-psiex0} at a rate $O(1/\sqrt{k})$.

\begin{proposition}\label{prop:opt-GMA}
Assume $\nu \in (0,1)$ and if $\omega^2 \geq \frac{\sigma^2}{\min(\nu,1-\nu)}$.
Consider $\bfpsi^{(k)}$ as in Algorithm~\ref{ALG:GMA-BMAX} (GMA-BMAX).
For any $\bflambda \in \Lambda^M$, the following oracle inequality holds
\begin{align}
\|\bfpsi^{(k)}-\bfeta\|_2^2
\leq &\ \nu \sum_{j=1}^M \lambda_j \|\bff_j-\bfeta\|_2^2  + (1-\nu)\left\|\kf_{\bflambda}- \bfeta\right\|_2^2
+2\omega^2 \cK(\bflambda,\bfpi\delta)
\label{INEQ:opt-GMA-dev}
\\
&+2\sqrt{\frac{16A_3}{A_1(k+3)}}\|\bfpsi_{X}(\omega^2,\nu)-\bfeta\|_2
+\frac{16A_3}{A_1(k+3)}
\nonumber
\end{align}
with probability at least $1-\delta$. Moreover,
\begin{align}
\E \|\bfpsi^{(k)}-\bfeta\|_2^2
\leq &\ \nu \sum_{j=1}^M \lambda_j \|\bff_j-\bfeta\|_2^2  + (1-\nu)\left\|\kf_{\bflambda}- \bfeta\right\|_2^2
\label{INEQ:opt-GMA-exp}
\\
&+ 2\omega^2 \cK(\bflambda,\bfpi)
+2\sqrt{\frac{16A_3}{A_1(k+3)}}\E\|\bfpsi_{X}(\omega^2,\nu)-\bfeta\|_2 + \frac{16A_3}{A_1(k+3)}.
\nonumber
\end{align}
\end{proposition}

From Proposition~\ref{prop:opt-GMA}, if $\omega^2 \geq \frac{\sigma^2}{\min(\nu,1-\nu)}$, for any $j=1,\ldots,M$ we have
\[
\|\bfpsi^{(k)}-\bfeta\|_2^2 \leq \|\bff_j-\bfeta\|_2^2+ 2\omega^2 \log\left(\frac{1}{\pi_j\delta}\right) + O(1/\sqrt{k}) \;
\]
with probability at least $1-\delta$, and
\[
\E\|\bfpsi^{(k)}-\bfeta\|_2^2 \leq \|\bff_j-\bfeta\|_2^2+ 2\omega^2 \log\left(\frac{1}{\pi_j}\right) + O(1/\sqrt{k}) \;.
\]
When $k \to \infty$, $\bfpsi^{(k)}$ achieves the optimal deviation bound. However, it does not imply optimal deviation bound for $\bfpsi^{(k)}$ with small $k$, while the greedy algorithms described in \cite{DaiZha11} (GMA) and \cite{DaiRigZhang12} (GMA-0) achieve optimal deviation bounds for small $k$ when $k\geq 2$. The advantage of GMA-BMAX is that the resulting estimator $\psi^{(k)}$
competes with any $\kf_{\bflambda}$ with $\bflambda \in \Lambda^M$ under the KL entropy, and such a result can be applied even for infinity dictionaries
containing functions indexed by continuous parameters, as long as the KL divergence $\cK(\bflambda,\bfpi)$ is well-defined (see relevant discussions in Section~\ref{sec:duality}).
On the other hand, the greedy estimators of \cite{DaiRigZhang12} for the $Q$-aggregation scheme can only deal with an upper bound of KL divergence
referred to as linear entropy (see Section~\ref{sec:duality}) that is not well-defined  for continuous dictionaries.
This means that GMA-BMAX is more generally applicable than the corresponding greedy algorithm GMA-0 in \cite{DaiRigZhang12}.

\subsection{Gradient Descent Algorithm (GD-BMAX)}

\begin{algorithm}[t]
\caption{Gradient Descent Algorithm (GD-BMAX)}
\label{ALG:GD}
\begin{algorithmic}
\REQUIRE Noisy observation $\bfY$, dictionary $\cH=\{f_1,\ldots,f_M\}$, prior $\bfpi \in \Lambda^M$, parameters $\nu, \omega^2$.
\ENSURE  Aggregate estimator  $\bfpsi^{(k)}$.
\STATE  Let $\bfpsi^{(0)}=0$.
\smallskip
\FOR{$k=1,2, \ldots$}
\STATE Choose step size $t_k\in(0,2/A_2)$ for $k>0$;
\STATE Let
\begin{align*}
\kf_{\bflambda^{(k-1)}}= \sum_{j=1}^M \lambda_j^{(k-1)} \bff_j,
\end{align*}
where $\bflambda^{(k-1)} \in\Lambda^M$ and
\begin{align}
\lambda_j^{(k-1)} \propto \pi_j \exp\Bigg(&-\frac{1}{2\omega^2}\|\bff_j-\bfY\|_2^2
+\frac{1-\nu}{2\omega^2}\|\bfpsi^{(k-1)}-\bff_j\|_2^2\Bigg);
\label{EQ:lam-GD-BMAX}
\end{align}
\STATE
$
\bfpsi^{(k)} = (1-t_k\frac{1-\nu}{\omega^2}) \bfpsi^{(k-1)} + t_k\frac{1-\nu}{\omega^2}\kf_{\bflambda^{(k-1)}}$;
\ENDFOR
\end{algorithmic}
\end{algorithm}

An alternative way solving the BMAX is to use the gradient descend method. The GD-BMAX algorithm is shown in Algorithm \ref{ALG:GD}.
The gradient is
\[
\nabla\log J(\bfpsi^{(k-1)})= \frac{\nabla J(\bfpsi^{(k-1)})}{J(\bfpsi^{(k-1)})}=\frac{1-\nu}{\omega^2}(\bfpsi^{(k-1)}-\kf_{\bflambda^{(k-1)}}) \;,
 \]
where $\bflambda^{(k-1)}\in\Lambda^M$ is defined as \eqref{EQ:lam-GD-BMAX}. In the $k$-th step, the update operation is
\begin{align*}
\bfpsi^{(k)}
= (1-t_k\frac{1-\nu}{\omega^2}) \bfpsi^{(k-1)} + t_k\frac{1-\nu}{\omega^2}\kf_{\bflambda^{(k-1)}}
= \bfpsi^{(k-1)}-t_k \nabla\log J(\bfpsi^{(k-1)}) \;.
\end{align*}
Therefore, Algorithm~\ref{ALG:GD} is a gradient decent algorithm with step size $t_k$. The following proposition shows the convergence of the GD-BMAX algorithm.

\begin{proposition}\label{PROP:GD-approxlogJ}
For $\bfpsi^{(k)}$ as defined in Algorithm~\ref{ALG:GD}, if we choose a fixed step size $t_k=s\in(0,2/A_2)$ for $k>0$ and $\{\bff_1,\ldots,\bff_M\}$ satisfy condition~\eqref{CON:f-l2normbound}, then
\begin{align}
\log J(\bfpsi^{(k)}) - \log J(\bfpsi_{X}(\omega^2,\nu))
\leq [1- 2A_1(s -(A_2/2)s^2)]^k  \left(\log J(\bfpsi^{(0)})- \log J(\bfpsi_{X}(\omega^2,\nu))\right).
\label{INEQ:GD-approxlogJ}
\end{align}
\end{proposition}

\begin{remark}\label{rema:1}
We may choose $t_k=s=1/A_2$ to minimize the righthand side of \eqref{INEQ:GD-approxlogJ} such that
\begin{align*}
\log J(\bfpsi^{(k)}) - \log J(\bfpsi_{X}(\omega^2,\nu))
\leq (1- A_1/A_2)^k  \left(\log J(\bfpsi^{(0)})- \log J(\bfpsi_{X}(\omega^2,\nu))\right).
\end{align*}
\end{remark}

Proposition~\ref{PROP:GD-approxlogJ} shows that the GD-BMAX algorithm converges at a geometric rate of $O(q^k)$ with $q=1- 2A_1(s -(A_2/2)s^2)$. Moreover, we have the following oracle inequality, which shows that the regret of the estimator $\bfpsi^{(k)}$ after running $k$ steps of GD-BMAX converges to that of $\bfpsi_{X}(\omega^2,\nu)$ in Corollary~\ref{COR:opt-psiex0} at a rate of $O(q^k)$.

\begin{proposition}
Assume $\nu \in (0,1)$ and if $\omega^2 \geq \frac{\sigma^2}{\min(\nu,1-\nu)}$, we have oracle inequality for any $\bflambda \in \Lambda^M$,
\begin{align}
\|\bfpsi^{(k)}-\bfeta\|_2^2
\leq &\ \nu \sum_{j=1}^M \lambda_j \|\bff_j-\bfeta\|_2^2  + (1-\nu)\left\|\kf_{\bflambda}- \bfeta\right\|_2^2
+ 2\omega^2 \cK(\bflambda,\bfpi) + 2\omega^2\log(1/\delta) 
\label{INEQ:opt-GD-dev}
\\
&+ 2\sqrt{L^2[1- 2A_1(s -(A_2/2)s^2)]^k}\|\bfpsi_{X}(\omega^2,\nu)-\bfeta\|_2 
+  L^2[1- 2A_1(s -(A_2/2)s^2)]^k \;,
\nonumber
\end{align}
with probability at least $1-\delta$. Moreover,
\begin{align}
\E \|\bfpsi^{(k)}-\bfeta\|_2^2
\leq &\ \nu \sum_{j=1}^M \lambda_j \|\bff_j-\bfeta\|_2^2  + (1-\nu)\left\|\kf_{\bflambda}- \bfeta\right\|_2^2
+ 2\omega^2 \cK(\bflambda,\bfpi) 
\label{INEQ:opt-GD-exp}
\\
&+ 2\sqrt{L^2[1- 2A_1(s -(A_2/2)s^2)]^k}\E\|\bfpsi_{X}(\omega^2,\nu)-\bfeta\|_2 
+  L^2[1- 2A_1(s -(A_2/2)s^2)]^k \;.
\nonumber
\end{align}
\label{PROP:opt-GD}
\end{proposition}
\noindent
From Proposition~\ref{PROP:opt-GD}, if $\omega^2 \geq \frac{\sigma^2}{\min(\nu,1-\nu)}$, for any $j=1,\ldots,M$ we have
\[
\|\bfpsi^{(k)}-\bfeta\|_2^2 \leq \|\bff_j-\bfeta\|_2^2+ 2\omega^2 \log\left(\frac{1}{\pi_j\delta}\right) + O(q^k) \;,
\]
with probability at least $1-\delta$ and
\[
\E\|\bfpsi^{(k)}-\bfeta\|_2^2 \leq \|\bff_j-\bfeta\|_2^2+ 2\omega^2 \log\left(\frac{1}{\pi_j}\right) + O(q^k) \;,
\]
for some constant $q=1- 2A_1(s -(A_2/2)s^2)\in(0,1)$.

Though the GD-BMAX algorithm does not give sparse output as the GMA-BMAX algorithm does, it has a faster geometric convergence compared to GMA-BMAX. Similar to Proposition~\ref{prop:opt-GMA}, the results in Proposition~\ref{PROP:opt-GD} do not imply optimal deviation bounds of $\bfpsi^{(k)}$ for small $k$ ($k<\infty$). 

The GD-BMAX algorithm can be naturally applied to continuous
dictionary case. When the dictionary is continuous or $M$ is large,
direct calculation of $\bflambda^{(k-1)} \in\Lambda^M$ in (\ref{EQ:lam-GD-BMAX}) is impractical and we can use the Markov Chain Monte Carlo (MCMC) sampling method \cite{MarRob07} to approximate $\bflambda^{(k-1)}$ for the $k$-th iteration in Algorithm~\ref{ALG:GD}.
\begin{algorithm}[t]
\caption{Metropolis-Hastings (MH) sampler for estimating $\kf_{\bflambda^{(k-1)}}$ at the $k$-th step of Algorithm~\ref{ALG:GD} }
\label{ALG:MH}
\begin{algorithmic}
\REQUIRE Noisy observation $\bfY$, dictionary $\cH=\{f_1,\ldots,f_M\}$, prior $\bfpi \in \Lambda^M$, parameters $\nu, \omega^2$, $(k-1)$-th step
estimator $\bfpsi^{(k-1)}$.
\ENSURE $\bfu_T^{(k-1)}$ as estimator of $\kf_{\bflambda^{(k-1)}}= \sum_{j=1}^M \lambda_j^{(k-1)} \bff_j$.
\STATE Initialize  $j(0)=0$;
\smallskip
\FOR{$t=1,\cdots, T_0+T$}
\STATE Generate $\tilde{j} \sim q(\cdot|j^{(t-1)})$ (e.g., $q$ can be chosen as a Gaussian distribution with mean $\bff_{j^{(t-1)}}$);
\STATE Compute
\[
\rho(j^{(t-1)},\tilde{j})=
\min\left( \frac{q(j^{(t-1)}|\tilde{j})\theta(\tilde{j})}{q(\tilde{j}|j^{(t-1)})\theta(j^{(t-1)})}, 1 \right) \;,
\]
where
\begin{align*}
\theta(j) = \pi_j \exp\Bigg(&-\frac{1}{2\omega^2}\|\bff_j-\bfY\|_2^2
+\frac{1-\nu}{2\omega^2}\|\bfpsi^{(k-1)}-\bff_j\|_2^2\Bigg);
\end{align*}
\STATE Generate a random variable $u\in[0,1]$ and let
\begin{align*}
j^{(t)}&=
\begin{cases}
\tilde{j} \;,&\text{if}\ \ \ u\leq\rho(j^{(t-1)},\tilde{j}); \\
j^{(t-1)} \;,& \text{otherwise}; \\
\end{cases}
\end{align*}
\ENDFOR
\STATE Calculate
\[
\bfu_T^{(k-1)}= \frac{1}{T}\sum_{t=T_0+1}^{T_0+T} \bff_{j^{(t)}};
\]
\end{algorithmic}
\end{algorithm}
A Metropolis-Hastings (MH) algorithm is given in Algorithm~\ref{ALG:MH}. The MH algorithm approximates $\kf_{\bflambda^{(k-1)}}$ with $\bfu_T^{(k-1)}$, so the resulting estimator $\bfpsi^{(k)}$ at the $k$-th step in Algorithm~\ref{ALG:GD} will have perturbations. Below we  provide a result showing how the perturbations from approximating $\kf_{\bflambda^{(k-1)}}$ would influence the convergence of $\log J(\bfpsi^{(k)})$ to $\log J(\bfpsi_{X}(\omega^2,\nu))$.

\begin{proposition}\label{PROP:GD-approxlogJ-MH}
Given $\bfY\in\R^n$, for all $k>0$, we assume that $\bfu_T^{(k-1)}$ in Algorithm~\ref{ALG:MH} satisfies
\begin{align}
& \E [\bfu_T^{(k-1)}|\bfpsi^{(k-1)}] = \kf_{\bflambda^{(k-1)}},  \\
& \|COV[\bfu_T^{(k-1)}|\bfpsi^{(k-1)}]\|_{\text{op}} \leq s^2,
\end{align}
where $\|\cdot\|_{\text{op}}$ is the matrix spectral norm. Then, we have
\begin{align*}
\E \left(\log J(\bfpsi^{(k)}) - \log J(\bfpsi_{X}(\omega^2,\nu))\right)
\leq &\ [1- 2A_1(s -(A_2/2)s^2)]^k  \left(\log J(\bfpsi^{(0)})- \log J(\bfpsi_{X}(\omega^2,\nu))\right) 
\\
&+ A_1 n s^2/2 \;.
\end{align*}
\end{proposition}

A variant of the GD-BMAX algorithm with continuous dictionary notation is also provided in Algorithm \ref{ALG:GD-CONT}, in which we assume the dictionary is parameterized by $w$ as $\cH_{\Omega}=\{\bff(w): \bff(w)\in\R^n \}$ and $ \Omega=\{w: w \in \R^d\}$.

\begin{algorithm}[!ht]
\caption{Gradient Descent Algorithm with Continuous Dictionary}
\label{ALG:GD-CONT}
\begin{algorithmic}
\REQUIRE Noisy observation $\bfY$, continuous dictionary $\cH_{\Omega}$, prior $\bfpi_{\Omega}$, parameters $\nu, \omega^2$.
\ENSURE  Aggregate estimator  $\bfpsi^{(k)}$.
\STATE  Let $\bfpsi^{(0)}=0$.
\smallskip
\FOR{$k=1,2, \ldots$}
\STATE Choose step size $t_k\in(0,2/A_2)$ for $k>0$;
\STATE Initialize  $w^{(0)}$;
\smallskip
\FOR{$t=1,\cdots, T_0+T$}
\STATE Generate $\tilde{w}\sim q(\cdot|w^{(t-1)})$;
\STATE Compute
\[
\gamma(w^{(t-1)},\tilde{w})=
\min\left( \frac{q(w^{(t-1)}|\tilde{w})\theta(\tilde{w})}{q(\tilde{w}|w^{(t-1)})\theta(w^{(t-1)})}, 1 \right) \;,
\]
where
\begin{align*}
\theta(w) = \pi(w) \exp\Bigg(&-\frac{1}{2\omega^2}\|\bff(w)-\bfY\|_2^2
+\frac{1-\nu}{2\omega^2}\|\bfpsi^{(k-1)}-\bff(w)\|_2^2\Bigg);
\end{align*}
\STATE Generate a random variable $u\in[0,1]$ and let
\begin{align*}
w^{(t)}&=
\begin{cases}
\tilde{w} \;,&\text{if}\ \ \ u\leq\gamma(w^{(t-1)},\tilde{w}); \\
w^{(t-1)} \;,& \text{otherwise}; \\
\end{cases}
\end{align*}
\ENDFOR
\STATE Calculate
\[
\bfu_T^{(k-1)}= \frac{1}{T}\sum_{t=T_0+1}^{T_0+T} \bff(w^{(t)});
\]
\STATE
\[
\bfpsi^{(k)} = (1-t_k\frac{1-\nu}{\omega^2}) \bfpsi^{(k-1)} + t_k\frac{1-\nu}{\omega^2}\bfu_T^{(k-1)};
\]
\ENDFOR
\end{algorithmic}
\end{algorithm}

\section{Experiments}
Although the contribution of this work is mainly theoretical, we include some simulations to illustrate the performance of the GMA-BMAX algorithm and GD-BMAX algorithm proposed for the BMAX method. We focus on the average performance of different algorithms and configurations.

The simulations will focus on discrete dictionary in order to compare the BMAX method with existing algorithms while the BMAX method can deal with continuous dictionary. 
Set $n=50$ and $M=500$. We identify a function $\bff$ with a vector $(f(x_1),\ldots, f(x_n))^\top \in \R^n$.
Let $\bfI_n$ denote the identity matrix of $\R^n$ and let $\Theta \sim \cN(0, \bfI_n)$ be a random vector, and define $\{\bff_1, \ldots, \bff_M\}$ as
\begin{equation}
\left\{
\begin{aligned}
\bff_j &= \Theta + s\cdot\bfzeta_j \;\;\; \text{for } 1\leq j \leq M_1 \;,\\
\bff_j &= \bfzeta_j  \;\;\; \text{for } M_1 < j \leq M \;,\\
\end{aligned}
\right.
\end{equation}
where $\bfzeta_j \sim \cN(0, \bfI_n) \; (j=1,\ldots,M)$ are independent random vectors.

Let $\Delta \sim \cN(0, \bfI_n)$ be a random vector. The regression function is defined by $\bfeta=\bff_1+0.5\Delta$. Note that typically $\bff_1$ will be the closest function to $\bfeta$ but not necessarily. The noise vector $\bfxi \sim \cN(0,\sigma^2\bfI_n)$ is independent of $\{\bff_1, \ldots, \bff_M\}$ and $\sigma=2$.

We define the oracle model (OM) $f_{k^\ast}$, where $k^\ast=\argmin_{j}\MSE(f_j)$. The model $f_{k^\ast}$ is clearly not a valid estimator because it depends on the unobserved $\eta$, however it can be used as a performance benchmark. The performance difference between an estimator $\hat\eta$ and the oracle model $f_{k^\ast}$ is measured by the \emph{regret} defined as:
\begin{equation}
R(\hat\eta) = \MSE(\hat\eta) - \MSE(f_{k^\ast}) \;.
\label{EQ:defregret}
\end{equation}

Since the target is $\bfeta=\bff_1 + 0.5 \Delta $, and $\bff_1$ and $\Delta$ are random Gaussian
vectors, the oracle model is likely $\bff_1$ (but it may not be $\bff_1$ due to the misspecification vector $\Delta$).
The noise $\sigma=2$ is relatively large, which implies a situation where the best convex aggregation does not outperform the oracle model.
This is the scenario  we considered here. For simplicity, all algorithms use a flat prior $\pi_j=1/M$ for all $j$. The experiment is performed with 100 replications.

One method compared is the STAR algorithm of \cite{Aud08}, which is optimal both in expectation and in deviation under the uniform prior. Mathematically, suppose $f_{k_1}$ is the empirical risk minimizer among functions in $\cH$, where
\begin{equation}
k_1 = \argmin_{j} \hMSE(f_j) \;,
\label{EQ:empmin}
\end{equation}
the STAR estimator $f^*$ is defined as
\begin{equation}
f^*= (1-\alpha^*)f_{k_1} + \alpha^* f_{k_2} \;,
\label{EQ:STAR1}
\end{equation}
where
\begin{equation}
(\alpha^*, k_2) = \argmin_{\alpha,j} \hMSE\big((1-\alpha)f_{k_1} + \alpha f_j \big), \ \ \alpha\in(0,1).
\label{EQ:STAR2}
\end{equation}

Another natural solution to solve the model averaging problem is to take the vector of weights $\bflambda^{\textsc{proj}}$ defined by
\begin{equation}
\label{EQ:proj}
\bflambdaproj \in \argmin_{\bflambda \in \Lambda^M }\hMSE(\ff_{\bflambda})\;,
\end{equation}
which minimizes the empirical risk. We call $\bflambdaproj$ the vector of \emph{projection weights} since the aggregate estimator $\kf_{\bflambdaproj}$ is the projection of $\bfY$ onto the convex hull of the $\bff_j$'s.

We compare the exponential weighted model averaging method denoted as EWMA. $Q$-aggregation with linear entropy (when $\rho(t)=1$) is also compared and will be solved by GMA-0 from \cite{DaiRigZhang12} (see Algorithm~\ref{ALG:gma0} below).

\begin{algorithm}[htp]
\caption{GMA-0 Algorithm}
\label{ALG:gma0}
\begin{algorithmic}
\REQUIRE Noisy observation $\bfY$, dictionary $\cH=\{f_1,\ldots,f_M\}$, prior $\bfpi \in \Lambda^M$, parameters $\nu, \beta$.
\ENSURE  Aggregate estimator  $\kf_{\bflambda^{(k)}}$.
\STATE  Let $\bflambda^{(0)}=0$,  $\kf_{\bflambda^{(0)}}=0$;
\smallskip
\FOR{$k=1, 2, \ldots$}
\STATE Set $\alpha_k=\frac{2}{k+1}$;
\STATE  $J^{(k)}=\argmin_j Q(\bflambda^{(k-1)}+ \alpha_k(\ee^{(j)} -  \bflambda^{(k-1)}))$;
\STATE  $\bflambda^{(k)}= \bflambda^{(k-1)}+ \alpha_k(\ee^{(J^{(k)})} - \bflambda^{(k-1)})$;
\ENDFOR
\end{algorithmic}
\end{algorithm}

We adopt flat priors $\pi=1/M$ ($j=1,\ldots,M$) for simplicity. From
the definition of $Q(\bflambda)$ \eqref{EQ:defQ}, it is easy to see
that, the minimizer of $Q(\bflambda)$ (when $\rho(t)=1$ with flat
prior) becomes $\bflambdaproj$ in \eqref{EQ:proj} by setting $\nu=0$,
so $\bflambdaproj$ is approximated by GMA-0 with $\nu=0$ and 200
iterations, and the projection algorithm is denoted by PROJ. 

We evaluate two versions of the GD-BMAX algorithm that one is exactly the algorithm depicted in Algorithm \ref{ALG:GD} and the other is to approximate $\kf_{\bflambda^{(k-1)}}$ in Algorithm \ref{ALG:GD} with the MH sampler in Algorithm \ref{ALG:MH}. We denote the latter variant as the GD-MH-BMAX algorithm. For the MH sampler, we use Gaussian distribution for $q(\cdot|\cdot)$ and set $T_0=T=500$.
The GMA-BMAX, GD-BMAX GD-MH-BMAX and GMA-0 algorithms are run for $K$ iterations up to $K=150$, with $\nu=1/2$. Parameter $\omega$ for GMA-BMAX, GD-BMAX, GD-MH-BMAX and EWMA is tuned by 10-fold cross-validation. Regrets of all algorithms defined in \eqref{EQ:defregret} are reported for comparisons.

In the following, we consider two scenarios. The first situation is when
the bases are not very correlated; in such case GMA-0 can perform better than
GMA-BMAX, GD-BMAX and GD-MH-BMAX because the former (which employs linear entropy) produces sparser estimators. The second situation is when the bases are highly correlated; in such case
GMA-BMAX, GD-BMAX and GD-MH-BMAX are superior than GMA-0 because the former algorithms (which employ strongly convex KL-entropy) give the clustered
basis functions similar weights while the GMA-0 tends to select one from the
clustered basis functions, which may lead to model selection error. The
correlated bases situation occurs in the continuous dictionary setting.

\subsection{Experiment 1: when $s=1$ and $M_1=50$, bases are not very correlated}
\begin{table*}[!ht]
\caption{Performance Comparison ($s=1$ and $M_1=50$)}
\label{table:exp1-1}
\begin{center}
\scalebox{0.78}{
\begin{tabular}{ccc} \hline
\textbf{STAR} & \textbf{PROJ} & \textbf{EWMA} \\ 
\hline
0.3895$\pm$0.3816 & 0.3953$\pm$0.2680 & 0.2823$\pm$0.5718  \\ 
\hline
\end{tabular}
}

\bigskip
\scalebox{0.78}{
\begin{tabular}{ccccccc} \hline
 & $k=1$  & $k=5$  & $k=15$  & $k=60$  & $k=100$  & $k=150$\\ 
 \hline
\textbf{GMA-BMAX} &  0.5726$\pm$0.7536 & 0.4196$\pm$0.4339 & 0.3480$\pm$0.3606 & 0.2840$\pm$0.3213 & 0.2745$\pm$0.3180 & 0.2690$\pm$0.3175 \\ 
\textbf{GD-BMAX} &  1.6187$\pm$0.3313 & 0.7231$\pm$0.2305 & 0.3003$\pm$0.3098 & 0.2603$\pm$0.3161 & 0.2602$\pm$0.3161 & 0.2602$\pm$0.3161 \\ 
\textbf{GD-MH-BMAX} & 1.6192$\pm$0.3311 & 0.7253$\pm$0.2328 & 0.3061$\pm$0.3127 & 0.2645$\pm$0.3107 & 0.2705$\pm$0.3237 & 0.2583$\pm$0.3058 \\
\textbf{GMA-0} & 0.3742$\pm$0.7804 & 0.2916$\pm$0.3691 & 0.2567$\pm$0.3291 & 0.2546$\pm$0.3289 & 0.2548$\pm$0.3292 & 0.2550$\pm$0.3287 \\ 
\hline
\end{tabular}
}
\end{center}
\end{table*}

\begin{table}[!ht]
\caption{Cumulative Frequency of Regret ($s=1$, $M_1=50$ and $k=150$)}
\label{table:exp1-2}
\begin{center}
\scalebox{0.78}{
\begin{tabular}{ccccccccc} \hline
Upper Boundary & 0.00 & 0.33 & 0.66 & 0.98 & 1.31 & 1.64 & 1.97 & 2.29 \\ 
\hline
\textbf{GMA-BMAX} &  17 & 69 & 85 & 95 & 100 & 100 & 100 & 100 \\ 
\textbf{GD-BMAX} & 13 & 70 & 85 & 96 & 100 & 100 & 100 & 100 \\
\textbf{GD-MH-BMAX} & 14 & 68 & 85 & 97 & 100 & 100 & 100 & 100 \\
\textbf{GMA-0} &  21 & 72 & 84 & 96 & 100 & 100 & 100 & 100 \\ 
\textbf{EWMA} &  39 & 78 & 82 & 84 & 88 & 95 & 99 & 100 \\ 
\hline
\end{tabular}
}
\end{center}
\end{table}

Table~\ref{table:exp1-1} compares the commonly used estimators, i.e., STAR, PROJ and EWMA, with GMA-BMAX, GD-BMAX, GD-MH-BMAX, and GMA-0. The regrets are reported using the format of ``$\text{mean} \pm \text{standard deviation}$''. Table~\ref{table:exp1-2} reports the cumulative frequency of GMA-BMAX, GD-BMAX, GD-MH-BMAX, GMA-0 and EWMA with 100 replicates and fixed iteration $k$. For each entry, we summarize the number of replicates with regrets which are smaller than or equal to the upper boundary value.

The results in Table~\ref{table:exp1-1} indicate that GMA-0 achieves the best performance as iteration $k$ increases. GMA-0 outperforms STAR, EWMA and PROJ after as small as $k=15$ iterations, which still gives a relatively sparse averaged model. This is consistent with Theorems~4.1 and 4.2 in \cite{DaiRigZhang12} which shows that GMA-0 has optimal bounds for small $k$ ($k\geq 2$).

GMA-BMAX, GD-BMAX and GD-MH-BMAX prefer dense model that assigns similar weights to similar candidates. 
It is easy to verify that $\bfpsi_{X}(\omega^2,\nu)=\kf_{\bflambda}$ with $\bflambda\in\Lambda^M$ defined as
\begin{align*}
\lambda_j \propto \pi_j \exp\left(-\frac{1}{2\omega^2}\|\bff_j-\bfY\|_2^2+\frac{1-\nu}{2\omega^2}\|\bfpsi_{X}(\omega^2,\nu)-\bff_j\|_2^2\right) \;,
\end{align*}
and thus two similar candidates $\bff_i$ and $\bff_j$ will have similar weights.

In contrast, GMA-0 prefers sparse model by selecting candidates less related to the estimator from previous iteration. Specifically, with flat prior $\bfpi$, the choice of $J^{(k)}$ in GMA-0 algorithm can be further simplified to
\begin{equation}
J^{(k)}=\argmin_j \left\{\|\bff_j-\bfY\|_2^2-(1-\nu)(1-\alpha_k)\|\kf_{\bflambda^{(k-1)}}-\bff_j\|_2^2 \right\} \;,
\label{EQ:gma0Jk-simp1}
\end{equation}
and thus at each iteration $k$ in the GMA-0 algorithm, the estimator $\bff_j$ is preferred if it is close to $\bfY$ while it is less correlated to the current aggregate estimator $\kf_{\bflambda^{(k-1)}}$ (because the minimization requires $\|\kf_{\bflambda^{(k-1)}}-\bff_j\|_2^2$ to be large while $\|\bff_j-\bfY\|_2^2$ being small).

In Experiment 1, the first 50 candidates $\{\bff_1, \ldots, \bff_M\}$ are closer to the truth $\bfeta=\bff_1+0.5\Delta$ than other candidate $\bff_j$ for $j>50$, yet they are not very correlated when $s=1$. Sparsity is preferred when correlations are not strong among the predicting features (in our experiment, the first 50 candidates), and GMA-0 is to output sparser estimator than GMA-BMAX, GD-BMAX and GD-MH-BMAX. Therefore, we would expect GMA-0 achieving smaller regret than GMA-BMAX, GD-BMAX and GD-MH-BMAX under this situation.
Although GMA-BMAX, GD-BMAX and GD-MH-BMAX are worse than GMA-0 when the bases are not very correlated, they beat EWMA, STAR and PROJ when the iteration $k$ is large enough.
\begin{figure*}[!ht]
\centering
\subfigure[]{
\includegraphics[scale=0.25]{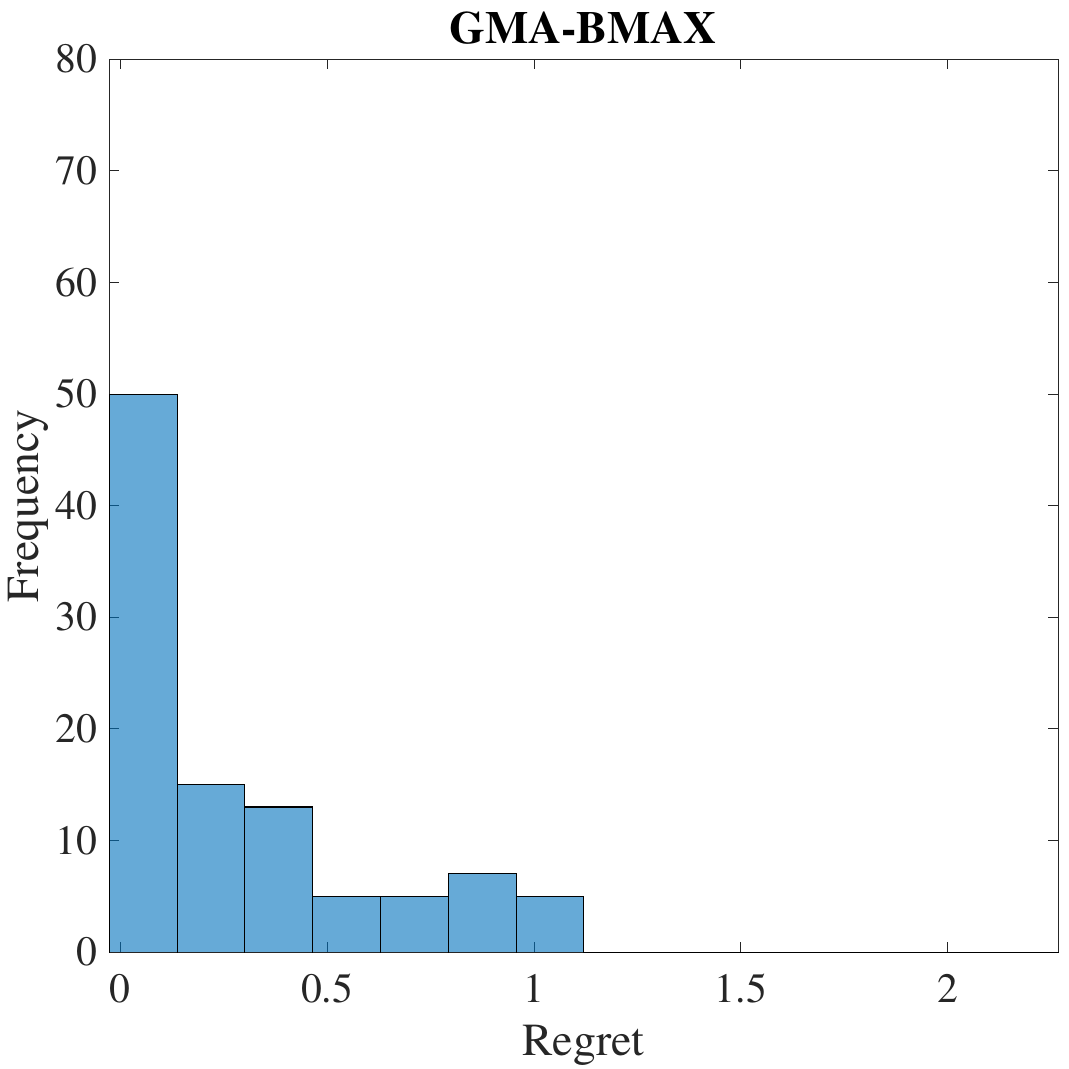}
}
\subfigure[]{
\includegraphics[scale=0.25]{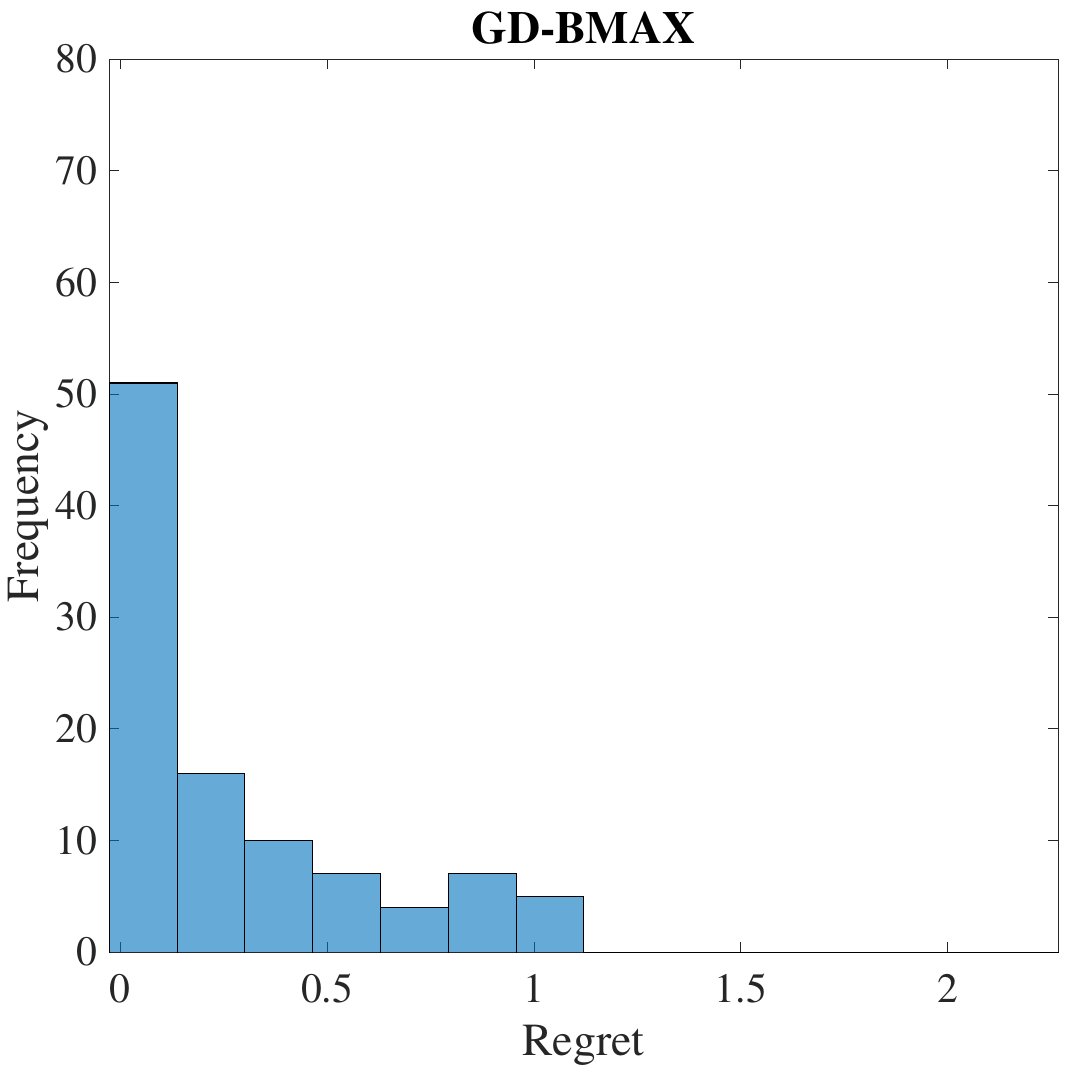}
}
\subfigure[]{
\includegraphics[scale=0.33]{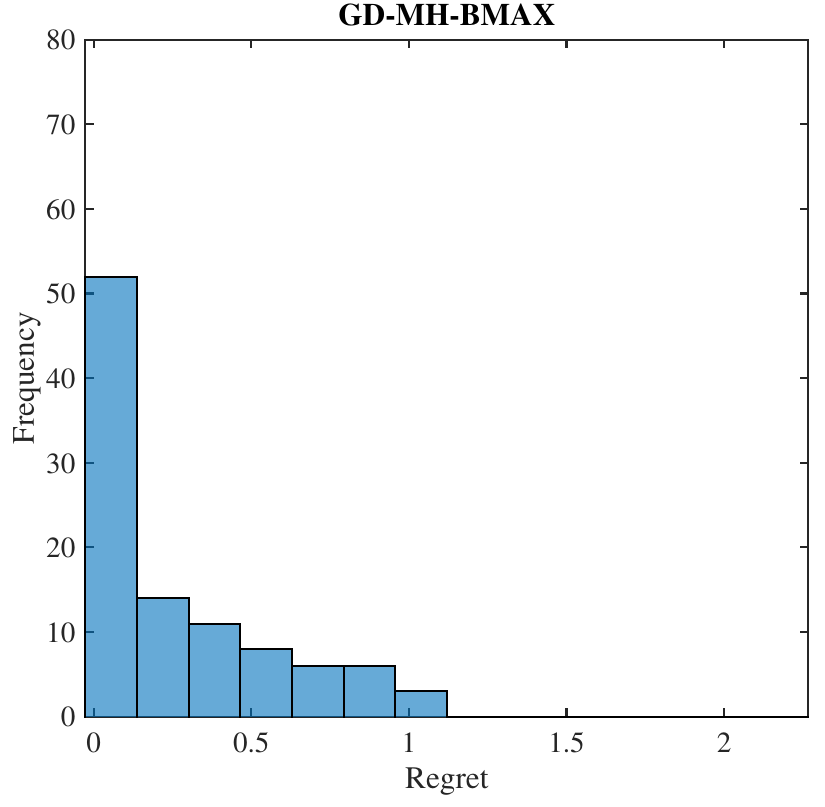}
}
\subfigure[]{
\includegraphics[scale=0.25]{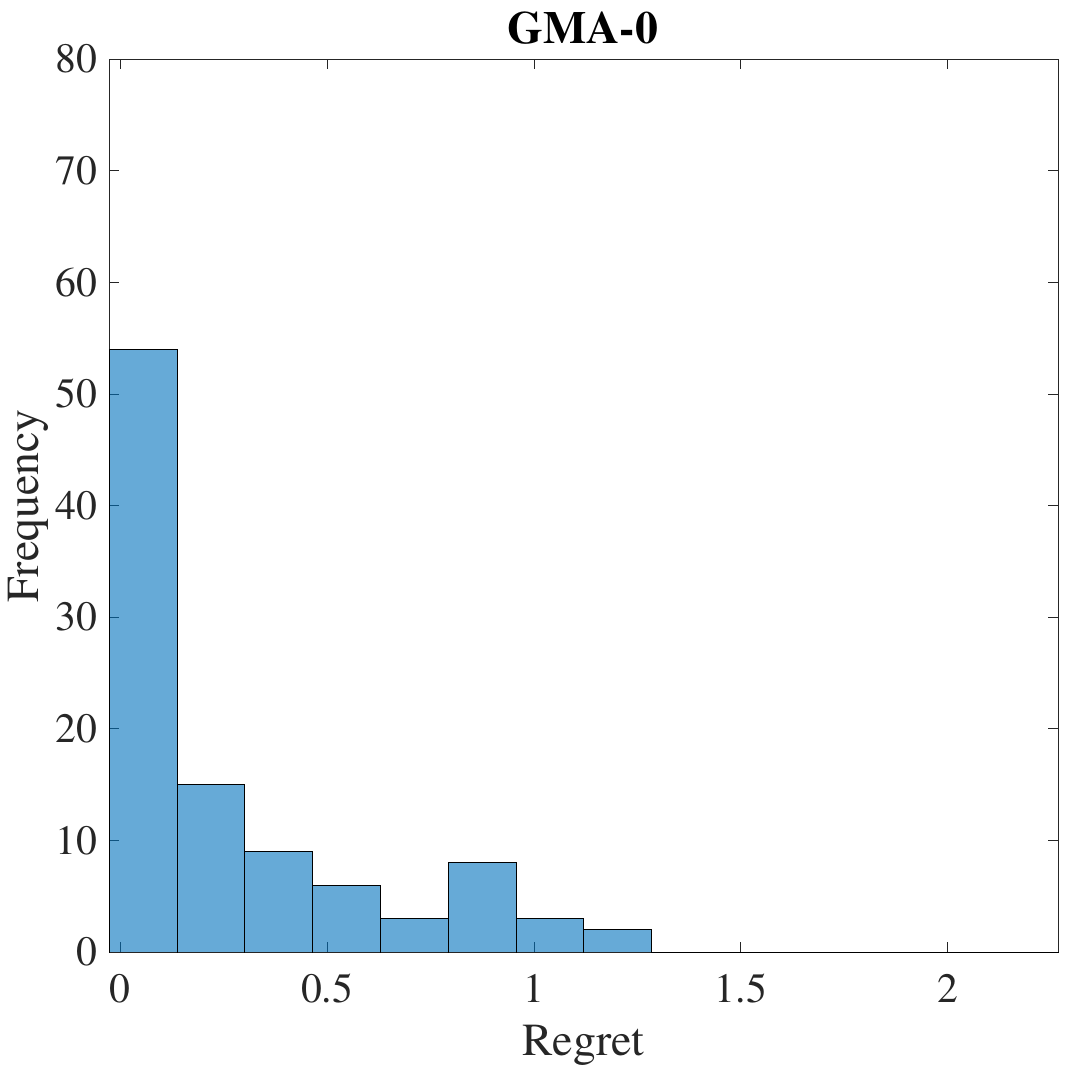}
}
\subfigure[]{
\includegraphics[scale=0.25]{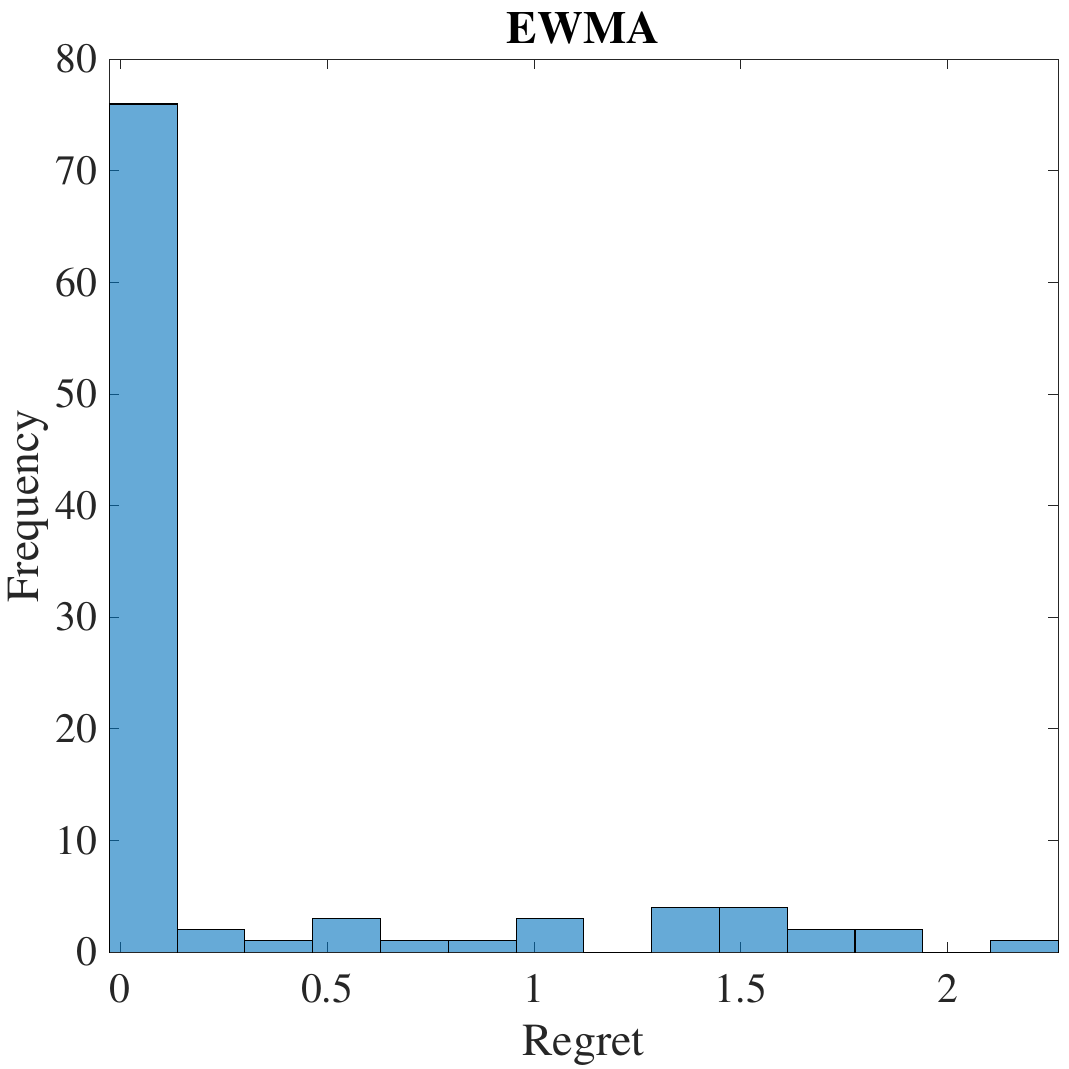}
}
\subfigure[]{
\includegraphics[scale=0.33]{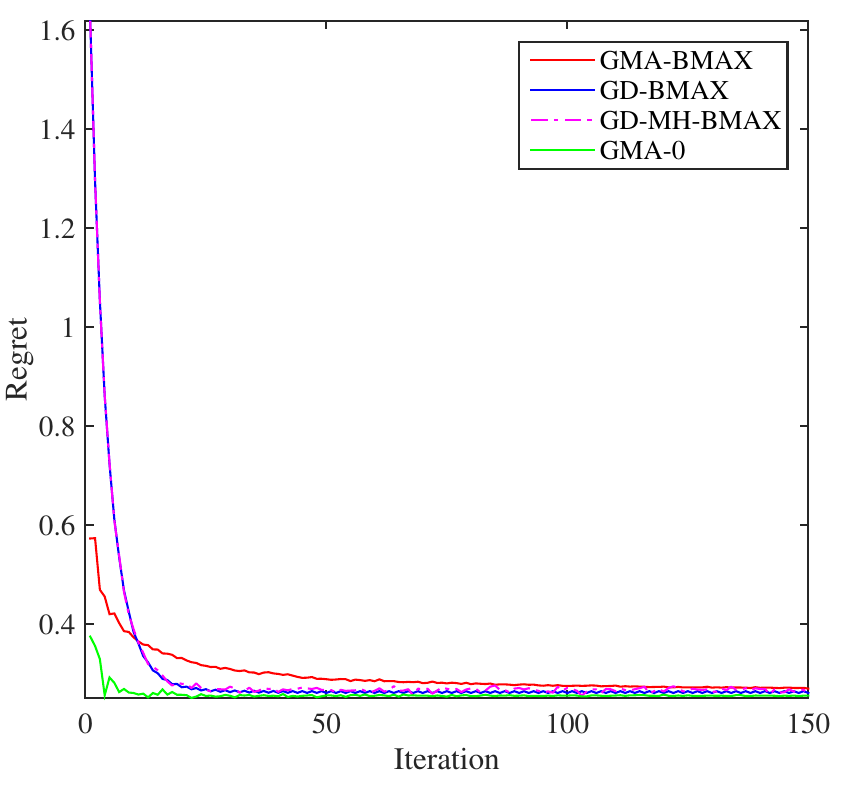}
}
\caption{(a-e) show the histograms of regrets for GMA-BMAX, GD-BMAX, GD-MH-BMAX, GMA-0 (with $k=150$) and EWMA; (f) reports the results of regrets $R(\bfpsi^{(k)})$ versus iterations $k$ under the case $s=1$ and $M_1=50$.}
\label{fig:exp1}
\end{figure*}

Figure~\ref{fig:exp1} compares the regrets of GMA-BMAX, GD-BMAX, GD-MH-BMAX, GMA-0 and EWMA. (a-e) illustrate the histograms of the regrets with 100 replicates. The corresponding cumulative frequencies are presented in Table~\ref{table:exp1-2}.
Since Proposition~\ref{prop:opt-GMA} indicates that the optimal deviation bound is obtained by $k \to \infty$, we pick $k=150$ for GMA-BMAX, GD-BMAX, GD-MH-BMAX and GMA-0 in order to make fair comparison. As the histograms show, although EWMA has the most replicates in which the regrets are close to zero, the distribution of the EWMA estimator is the most dispersive with many extreme values compared to other methods. The performance is consistent with~\cite{DalTsy07, DalTsy08} and \cite{DaiRigZhang12} which state that the EWMA estimator is optimal in expectation but sub-optimal in deviation. Therefore, we would expect GMA-BMAX, GD-BMAX, GD-MH-BMAX and GMA-0 enjoy more concentrated distribution than EWMA because they are also optimal in deviation. 
(f) shows the convergence of GMA-BMAX, GD-BMAX, GD-MH-BMAX and
GMA-0. We observe that GD-BMAX and GD-MH-BMAX show faster convergence
compared to GMA-BMAX, which is consistent with
Proposition~\ref{PROP:opt-GD}. Moreover, the GD-MH-BMAX algorithm
approximates GD-BMAX well with small perturbations, which is consistent with Proposition~\ref{PROP:GD-approxlogJ-MH}.

Note that GMA-BMAX and GMA-0 produce different estimators after the first iteration ($k=1$). GMA-BMAX selects $j\in\{1,\ldots,M\}$ that minimizes $\log J(\bff_j)$, while \mbox{GMA-0} selects $j\in\{1,\ldots,M\}$ that minimizes $Q(\bff_j)$ and the output of the first stage is actually the empirical risk minimizer $\bff_{k_1}$ where $k_1=\argmin_{j}\hMSE(f_j)$. Moreover, since sparse model is preferred in this scenario, GMA-0 has smaller regret than GMA-BMAX, GD-BMAX and GD-MH-BMAX, and all of them converge fast within a few iterations.

\subsection{Experiment 2: when $s=\sigma/\|\bfzeta_j\|$ and $M_1=M$, bases are all highly correlated}
\begin{table*}[!ht]
\caption{Performance Comparison($s=\sigma/\|\bfzeta_j\|$ and $M_1=M$)}
\label{table:exp2-1}
\begin{center}
\scalebox{0.78}{
\begin{tabular}{ccc} \hline
\textbf{STAR} & \textbf{PROJ} & \textbf{EWMA} \\ 
\hline
0.0587$\pm$0.0338 & 0.0495$\pm$0.0322 & 0.0452$\pm$0.0307  \\ 
\hline
\end{tabular}
}

\bigskip
\scalebox{0.78}{
\begin{tabular}{ccccccc} \hline
 & $k=1$  & $k=5$  & $k=15$  & $k=60$  & $k=100$  & $k=150$\\ 
 \hline
\textbf{GMA-BMAX} &  0.0929$\pm$0.0381 & 0.0503$\pm$0.0300 & 0.0431$\pm$0.0285 & 0.0408$\pm$0.0282 & 0.0407$\pm$0.0283 & 0.0406$\pm$0.0283 \\ 
\textbf{GD-BMAX} &  0.7817$\pm$0.2207 & 0.2525$\pm$0.1041 & 0.0501$\pm$0.0317 & 0.0406$\pm$0.0284 & 0.0406$\pm$0.0284 & 0.0406$\pm$0.0284 \\ 
\textbf{GD-MH-BMAX} & 0.7414$\pm$0.2126 & 0.1996$\pm$0.0876 & 0.0435$\pm$0.0291 & 0.0393$\pm$0.0273 & 0.0382$\pm$0.0273 & 0.0384$\pm$0.0284 \\
\textbf{GMA-0} & 0.0904$\pm$0.0380 & 0.0650$\pm$0.0379 & 0.0630$\pm$0.0371 & 0.0621$\pm$0.0364 & 0.0622$\pm$0.0364 & 0.0620$\pm$0.0364 \\ 
\hline
\end{tabular}
}
\end{center}
\end{table*}

\begin{table}[!ht]
\caption{Cumulative Frequency of Regret ($s=\sigma/\|\bfzeta_j\|$, $M_1=M$ and $k=150$)}
\label{table:exp2-2}
\begin{center}
\scalebox{0.78}{
\begin{tabular}{ccccccccc} \hline
Upper Boundary & 0.000 & 0.026 & 0.051 & 0.077 & 0.103 & 0.128 & 0.154 & 0.180 \\ 
\hline
\textbf{GMA-BMAX} &  5 & 31 & 69 & 89 & 97 & 99 & 100 & 100 \\ 
\textbf{GD-BMAX} & 6 & 31 & 69 & 88 & 97 & 99 & 100 & 100 \\
\textbf{GD-MH-BMAX} & 5 & 36 & 71 & 90 & 97 & 99 & 100 & 100 \\
\textbf{GMA-0} &  4 & 17 & 40 & 67 & 89 & 95 & 99 & 100 \\ 
\textbf{EWMA} &  5 & 26 & 59 & 86 & 96 & 98 & 100 & 100 \\ 
\hline
\end{tabular}
}
\end{center}
\end{table}

\begin{figure*}[ht]
\centering
\subfigure[]{
\includegraphics[scale=0.25]{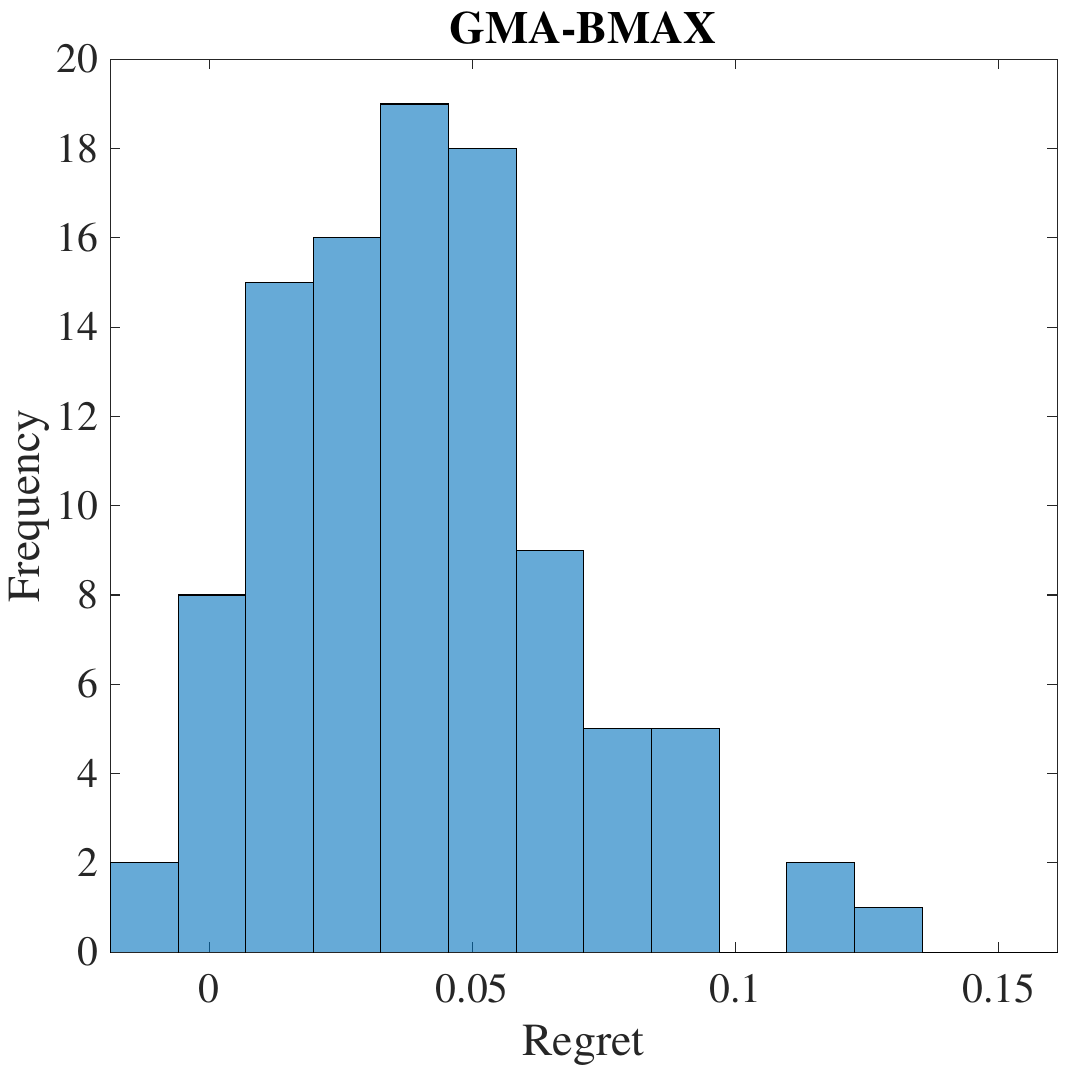}
}
\subfigure[]{
\includegraphics[scale=0.25]{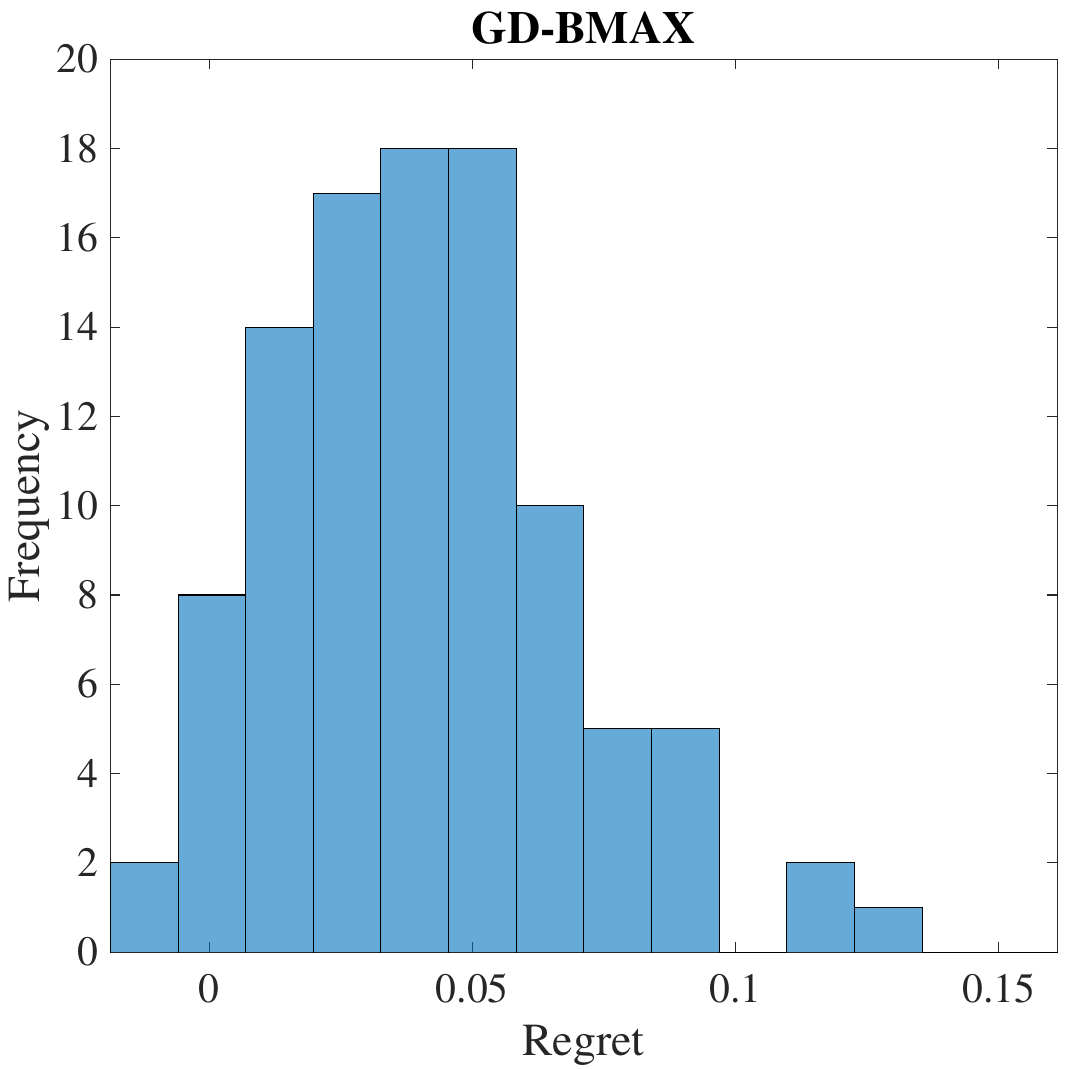}
}
\subfigure[]{
\includegraphics[scale=0.33]{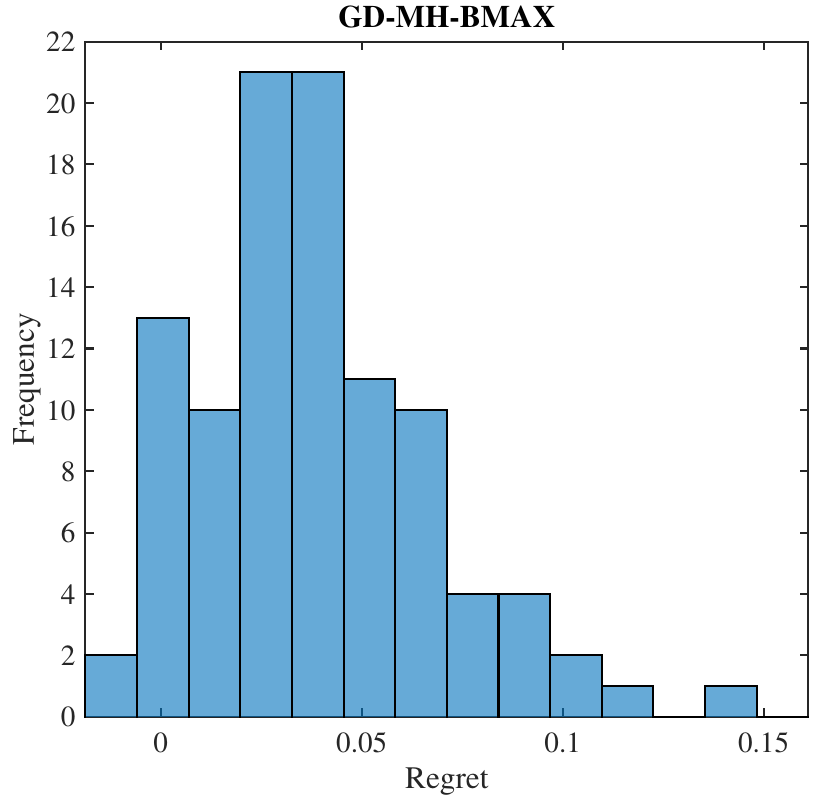}
}
\subfigure[]{
\includegraphics[scale=0.25]{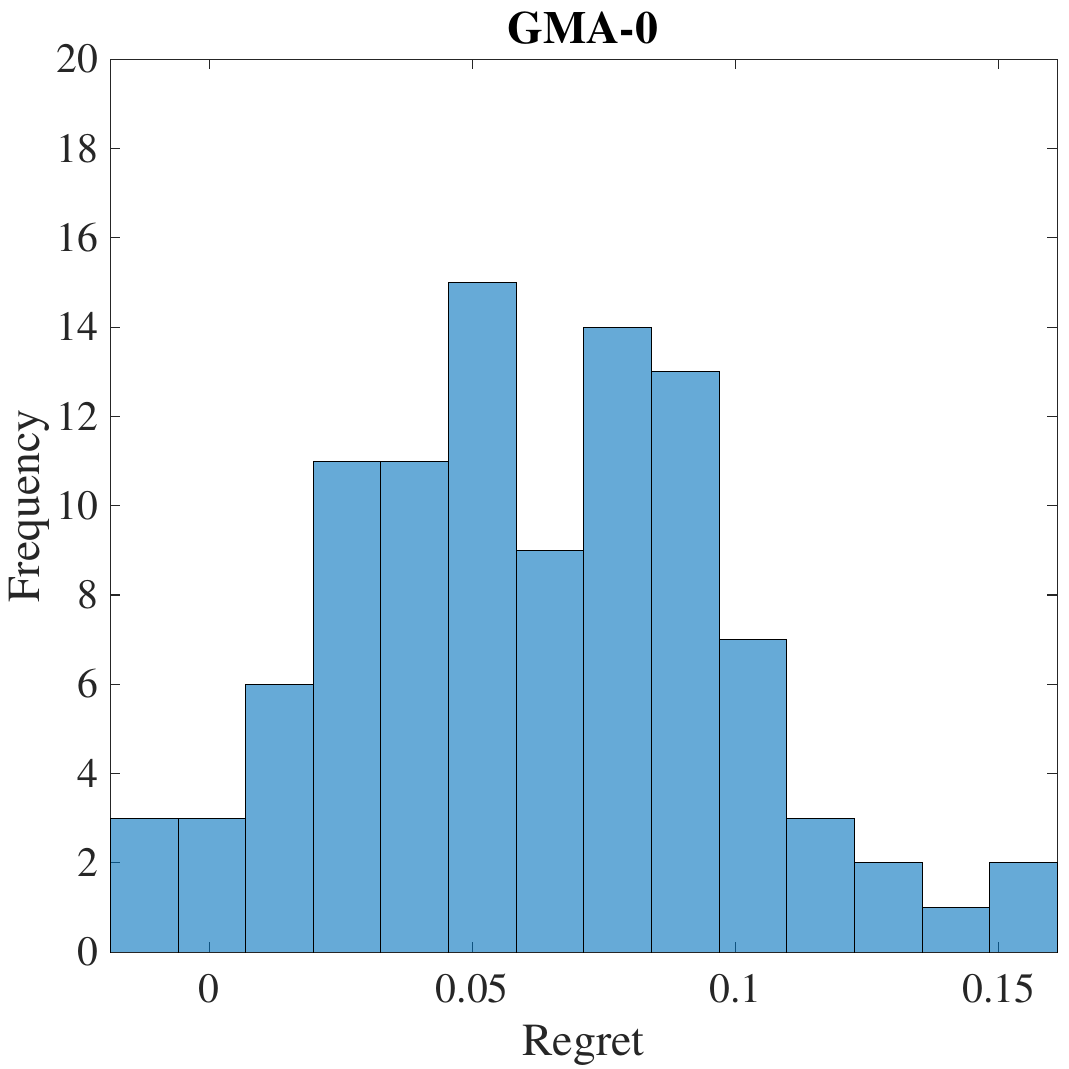}
}
\subfigure[]{
\includegraphics[scale=0.25]{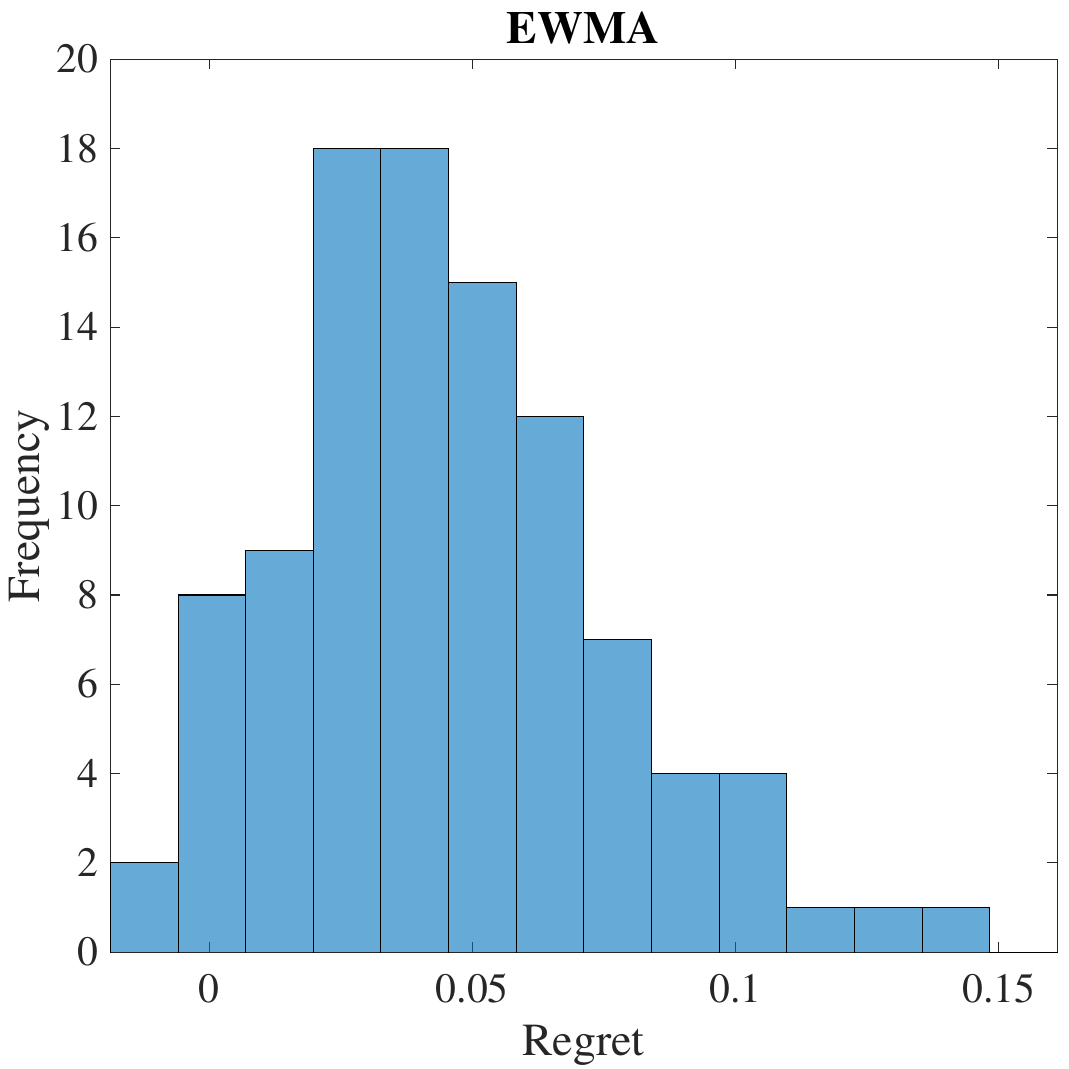}
}
\subfigure[]{
\includegraphics[scale=0.33]{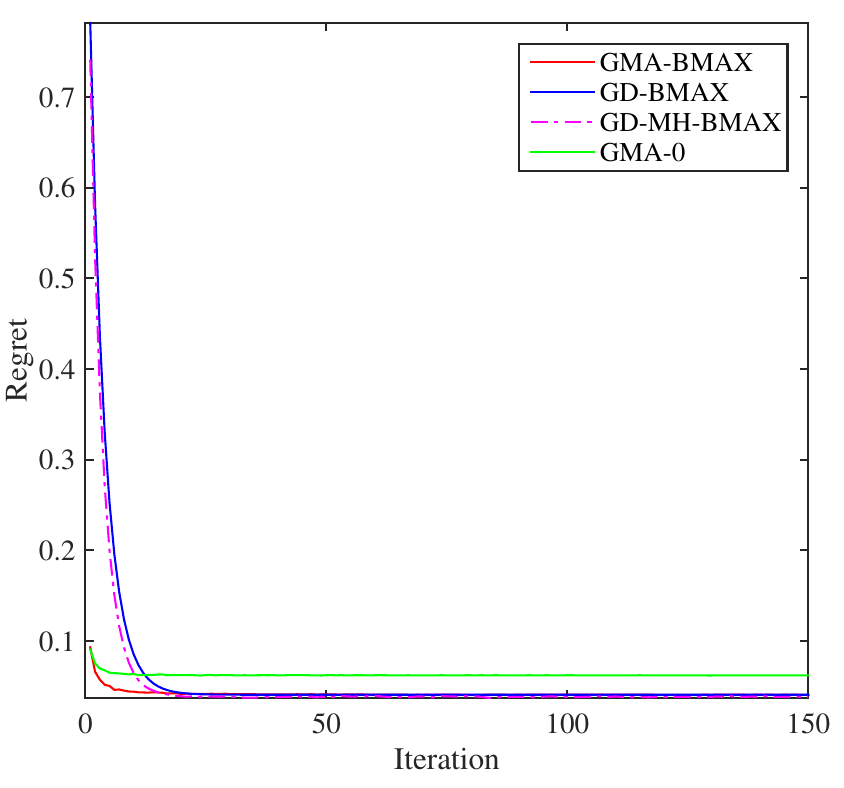}
}
\caption{(a-e) show the histograms of regrets for GMA-BMAX, GD-BMAX, GD-MH-BMAX, GMA-0 (with $k=150$) and EWMA; (f) reports the result of regrets $R(\bfpsi^{(k)})$ versus iterations $k$ under the case $s=\sigma/\|\bfzeta_j\|$ and $M_1=M$.}
\label{fig:exp2}
\end{figure*}

In Experiment 2, we define regression function as $\bfeta=\Theta+0.5\Delta$ which is slightly different from Experiment 1. The results in Table~\ref{table:exp2-1} indicate that GMA-BMAX, GD-BMAX and GD-MH-BMAX perform better than GMA-0 as iteration $k$ increases, and GMA-BMAX, GD-BMAX and GD-MH-BMAX also beat STAR, PROJ and EWMA when $k$ is large enough.

In this experiment, all of the candidates $\{\bff_1, \ldots, \bff_M\}$ are close to the truth $\bfeta=\Theta+0.5\Delta$, and they are highly correlated when $s$ is small. GMA-BMAX, GD-BMAX and GD-MH-BMAX tend to assign similar weights to similar candidates, while GMA-0 tends to exclude other correlated candidates once one has been selected. Therefore, GMA-BMAX, GD-BMAX and GD-MH-BMAX will average over those candidates with similar weights resulting less variance (also less bias due to the design), while GMA-0 will have high variance because it only selects one of the candidates. Moreover, similar to GMA-BMAX, GD-BMAX and GD-MH-BMAX, EWMA also prefer dense model by assigning similar weights to similar candidates. However, the EWMA estimator assigns weights defined as
\begin{eqnarray*}
\lambda_j \propto \pi_j \exp\left(-\frac{1}{2\omega^2}\|\bff_j-\bfY\|_2^2\right) \;,
\end{eqnarray*}
while the weights in GMA-BMAX, GD-BMAX and GD-MH-BMAX are defined as
\begin{align*}
\lambda_j \propto \pi_j \exp\left(-\frac{1}{2\omega^2}\|\bff_j-\bfY\|_2^2+\frac{1-\nu}{2\omega^2}\|\bfpsi_{X}(\omega^2,\nu)-\bff_j\|_2^2\right) \;,
\end{align*}
where $\bfpsi_{X}(\omega^2,\nu)=\kf_{\bflambda}$ with $\bflambda\in\Lambda^M$.
Note that for all $j$, $\|\bff_j-\bfY\|_2^2$ are roughly equal to each other under this scenario. That is, the EWMA estimator becomes the average of all candidates. However, the GMA-BMAX, GD-BMAX and GD-MH-BMAX estimators are still weighed averages of all bases, and the weights are adjusted by the extra term $\|\bfpsi_{X}(\omega^2,\nu)-\bff_j\|_2^2$. Therefore, we would hope GMA-BMAX, GD-BMAX and GD-MH-BMAX have smaller variances than EWMA.

Figure~\ref{fig:exp2} compares the regrets of GMA-BMAX, GD-BMAX, GD-MH-BMAX, GMA-0 and EWMA. (a-e) summarize the histograms of the regrets, and the corresponding cumulative frequencies are represented in Table~\ref{table:exp2-2}. GMA-BMAX, GD-BMAX and GD-MH-BMAX have the most concentrated distributions, because they are optimal both in expectation and in deviation. (f) illustrates the convergence of GMA-BMAX, GD-BMAX, GD-MH-BMAX and GMA-0. All these methods converge within a few iterations. As expected, GMA-BMAX, GD-BMAX and GD-MH-BMAX achieve lower regret than GMA-0 when the basis functions are correlated.

\section{Conclusion}

This paper introduces a new formulation for deviation optimal model averaging which we refer to as BMAX. It is motivated by Bayesian theoretical considerations with an appropriately defined exponentiated least squares loss. Moreover we established a primal-dual relationship of this
estimator and the $Q$-aggregation scheme (with KL entropy) by \cite{DaiRigZhang12}. This relationship not only establishes a natural Bayesian
interpretation for $Q$-aggregation but also leads to new numerical algorithms for model aggregation that are suitable for the
continuous dictionary setting where some basis functions are highly correlated. The new formulation and its relationship to $Q$-aggregation provides deeper understanding of deviation
optimal model averaging procedures.


%

\appendix
\section{Proofs}

\subsection{Proof of Lemma~\ref{LEM:logJ-2ord-convexity}}
Define $\bflambda\in\Lambda^M$ as
\begin{eqnarray*}
\lambda_j \propto \pi_j \exp\left(-\frac{1}{2\omega^2}\|\bff_j-\bfY\|_2^2+\frac{1-\nu}{2\omega^2}\|\bfpsi-\bff_j\|_2^2\right).
\end{eqnarray*}

It follows that
\begin{eqnarray*}
\frac{\nabla J(\bfpsi)}{J(\bfpsi)} = \frac{1-\nu}{\omega^2} (\bfpsi - \kf_{\bflambda}),
\end{eqnarray*}
and
\begin{align*}
\frac{\nabla^2 J(\bfpsi)}{J(\bfpsi)}
=\sum_{j=1}^M \lambda_j \left(\left(\frac{1-\nu}{\omega^2}\right)^2 (\bfpsi-\bff_j)(\bfpsi-\bff_j)^\top + \left(\frac{1-\nu}{\omega^2}\right) \bfI_n \right) .
\end{align*}

Then we have
\begin{align*}
\nabla^2 \log J(\bfpsi)
&=\frac{(\nabla^2 J(\bfpsi))J(\bfpsi)-(\nabla J(\bfpsi))(\nabla J(\bfpsi))^\top}{J^2(\bfpsi)}   \\
&=\sum_{j=1}^M \lambda_j \left(\left(\frac{1-\nu}{\omega^2}\right)^2 (\bfpsi-\bff_j)(\bfpsi-\bff_j)^\top+ \left(\frac{1-\nu}{\omega^2}\right) \bfI_n \right) 
-\left(\frac{1-\nu}{\omega^2}\right)^2 (\bfpsi - \kf_{\bflambda})(\bfpsi - \kf_{\bflambda})^\top \\
&=\left(\frac{1-\nu}{\omega^2}\right) \bfI_n +
\sum_{j=1}^M \lambda_j \left(\frac{1-\nu}{\omega^2}\right)^2 (\kf_{\bflambda}-\bff_j)(\kf_{\bflambda}-\bff_j)^\top.
\end{align*}
Therefore, $\nabla^2 \log J(\bfpsi) \geq \left(\frac{1-\nu}{\omega^2}\right) \bfI_n$.

With the assumption that $\|\bff_j\|_2 \leq L$ for all $j$, we have
\begin{align*}
\sum_{j=1}^M \lambda_j (\kf_{\bflambda}-\bff_j)(\kf_{\bflambda}-\bff_j)^\top
=\sum_{j=1}^M \lambda_j \bff_j\bff_j^\top - \kf_{\bflambda}\kf_{\bflambda}^\top
\leq \sum_{j=1}^M \lambda_j \bff_j\bff_j^\top
\leq \sum_{j=1}^M \lambda_j L^2 \bfI_n
= L^2 \bfI_n .
\end{align*}
It follows that
$\nabla^2 \log J(\bfpsi) \leq \left(\left(\frac{1-\nu}{\omega^2}\right)+ \left(\frac{1-\nu}{\omega^2}\right)^2 L^2 \right)\bfI_n $.
\epr

\subsection{Proof of Lemma~\ref{LEM:Q-S-T}}
\begin{lemma}
For any $\bflambda \in \Lambda^M$, real numbers $\{x_j\}_{j=1}^M$ and some constant $a > 0$, we have
\begin{eqnarray*}
\sum_{j=1}^M \lambda_j x_j - a \cK(\bflambda,\bfpi) \leq a \log\left( \sum_{j=1}^M \pi_j e^{x_j/a} \right) \;,
\end{eqnarray*}
where the equality is obtained when $(x_j/a) - \log(\lambda_j/\pi_j)$ is a constant for $1\leq j \leq M$.
\label{PROP:jensen}
\end{lemma}

\noindent
\emph{Proof:}
The result follows directly from Jensen's Inequality as
\begin{align*}
\exp \left(\sum_{j=1}^M \lambda_j ((x_j/a) - \log(\lambda_j/\pi_j)) \right)
\leq \sum_{j=1}^M \lambda_j \exp \left(  (x_j/a) - \log(\lambda_j/\pi_j) \right)
= \sum_{j=1}^M \pi_j e^{x_j/a} .
\end{align*}
\epr

Now by setting $x_j=-\nu\|\bff_j-\bfh\|_2^2$ and $a=2\omega^2$ in
Lemma~\ref{PROP:jensen}, we obtain
\begin{align*}
&\min_{\bflambda \in \Lambda^M} \left( \nu \sum_{j=1}^M \lambda_j\|\bff_j-\bfh\|_2^2+
    2\omega^2\cK(\bflambda,\bfpi) \right)
  -\frac{\nu}{1-\nu}\|\bfh-\bfY\|_2^2 \\
&=
 -\frac{\nu}{1-\nu}\|\bfh-\bfY\|_2^2-2\omega^2\log\left(\sum_{j=1}^M\pi_j
     e^{-\nu\|\bff_j-\bfh\|_2^2/2\omega^2}\right)  ,
\end{align*}
which implies that
\[
T(\bfh) = \min_{\bflambda \in \Lambda^M} S(\bflambda,\bfh) .
\]
In addition, it is easy to verify that 
\[
Q(\bflambda) = \max_{\bfh \in \R^n} S(\bflambda,\bfh) , 
\]
where the minimum is achieved at $\bfh= \kf_{\bflambda} $.

Now let $\hat{\bfh}$ be the maximizer of $T(\bfh)$ in \eqref{EQ:defT}, then by setting the derivative of \eqref{EQ:defT} to zero, it is easy to
observe that there exists a corresponding $\hat{\bflambda}$ so that $(\hat{\bflambda},\hat{\bfh}) \in A\cap B$. This means that
$A \cap B \neq \emptyset$.

Now consider any $(\bflambda^{0},\bfh^{0}) \in A\cap B$. We have
\begin{eqnarray*}
Q(\bflambda^{0}) \geq \min_{\bflambda \in \Lambda^M} Q(\bflambda) = \min_{\bflambda \in \Lambda^M} \max_{\bfh \in \R^n} S(\bflambda,\bfh) \geq \max_{\bfh \in \R^n} \min_{\bflambda \in \Lambda^M} S(\bflambda,\bfh) .
\end{eqnarray*}
The third inequality is the well-known weak duality (e.g., Lemma~36.1 in~\cite{Rockafellar97}).

Also we have
\begin{eqnarray*}
\max_{\bfh \in \R^n} \min_{\bflambda \in \Lambda^M} S(\bflambda,\bfh) 
= \max_{\bfh \in \R^n} T(\bfh) = T(\hat\bfh) \geq T(\bfh^{0}) \;.
\end{eqnarray*}
We thus have
\begin{align*}
Q(\bflambda^{0})
\geq \min_{\bflambda \in \Lambda^M} Q(\bflambda) = \min_{\bflambda \in \Lambda^M} \max_{\bfh \in \R^n} S(\bflambda,\bfh) 
\geq \max_{\bfh \in \R^n} \min_{\bflambda \in \Lambda^M} S(\bflambda,\bfh) = \max_{\bfh \in \R^n} T(\bfh) \geq T(\bfh^{0}) .
\end{align*}

Our target is now to prove $Q(\bflambda^{0})=T(\bfh^{0})$.
Since $(\bflambda^{0},\bfh^{0}) \in A\cap B$, we have
\begin{eqnarray*}
  \left\{
   \begin{aligned}
\bfh^{0} &= \frac{1}{\nu}\bfY-\frac{1-\nu}{\nu}\bff_{\bflambda^{0}} , \\
\lambda^{0}_j &= \frac{\exp\left(-\frac{\nu}{2\omega^2}\|\bff_j-\bfh^{0}\|_2^2\right)\pi_j}{\sum_{i=1}^M \exp\left(-\frac{\nu}{2\omega^2}\|\bff_i-\bfh^{0}\|_2^2\right)\pi_i} .
   \end{aligned}
  \right.
\end{eqnarray*}

It follows that for all $j$:
\begin{eqnarray*}
\sum_{i=1}^M \exp\left(-\frac{\nu}{2\omega^2}\|\bff_i-\bfh^{0}\|_2^2\right)\pi_i = \frac{\exp\left(-\frac{\nu}{2\omega^2}\|\bff_j-\bfh^{0}\|_2^2\right)\pi_j}{\lambda^{0}_j} ,
\end{eqnarray*}
which implies that
\begin{align*}
\log \left(\sum_{i=1}^M \exp\left(-\frac{\nu}{2\omega^2}\|\bff_i-\bfh^{0}\|_2^2\right)\pi_i \right)
&=-\frac{\nu}{2\omega^2}\|\bff_j-\bfh^{0}\|_2^2 -\log(\lambda^{0}_j/\pi_j)
\\
&=\sum_{i=1}^M \lambda^{0}_i \left( -\frac{\nu}{2\omega^2}\|\bff_i-\bfh^{0}\|_2^2 -\log(\lambda^{0}_i/\pi_i) \right) \;,
\end{align*}
where the second equation is from summing up two sides of the first equation with weight $\lambda^{0}_i$ over $i=1,\ldots,M$.

Plug back into $T(\bfh^{0})$, we obtain
\begin{align*}
T(\bfh^{0})
&=-\frac{\nu}{1-\nu}\|\bfh^{0}-\bfY\|_2^2
-2\omega^2 \Bigg[ \sum_{i=1}^M \lambda^{0}_i \Big( -\frac{\nu}{2\omega^2}\|\bff_i-\bfh^{0}\|_2^2
-\log(\lambda^{0}_i/\pi_i) \Big) \Bigg] 
\\
&=-\frac{\nu}{1-\nu}\|\bfh^{0}-\bfY\|_2^2 + \nu \sum_{i=1}^M \lambda^{0}_i \|\bff_i-\bfh^{0}\|_2^2 + 2\omega^2\cK(\bflambda^{0},\bfpi) 
\\
&=\|\bff_{\bflambda^{0}} -\bfY\|_2^2 + \nu \sum_{i=1}^M \lambda^{0}_i \|\bff_i-\bff_{\bflambda^{0}}\|_2^2 + 2\omega^2\cK(\bflambda^{0},\bfpi) 
\\
&=Q(\bflambda^{0}) \;,
\end{align*}
where the third equality is obtained by plugging in $\bfh^{0} = \frac{1}{\nu}\bfY-\frac{1-\nu}{\nu}\bff_{\bflambda^{0}}$.
Therefore, 
\begin{align*}
Q(\bflambda^{0})
=\min_{\bflambda \in \Lambda^M} Q(\bflambda) = \min_{\bflambda \in \Lambda^M} \max_{\bfh \in \R^n} S(\bflambda,\bfh) 
=\max_{\bfh \in \R^n} \min_{\bflambda \in \Lambda^M} S(\bflambda,\bfh) = \max_{\bfh \in \R^n} T(\bfh) = T(\bfh^{0}) .
\end{align*}

Since $Q(\cdot)$ is strictly convex and $T(\cdot)$ is strictly concave, we have
$\bfh^{0}=\hat\bfh$ is the unique solution of $\max_{\bfh} T(\bfh)$,
and $\bflambda^{0}=\bflambda$ is the unique solution of $\min_{\bflambda} Q(\bflambda)$. 
Using $\bfh^{0} = \frac{1}{\nu}\bfY-\frac{1-\nu}{\nu}\bff_{\bflambda^{0}}$, we have
\begin{eqnarray*}
\hat\bfh = \frac{1}{\nu}\bfY-\frac{1-\nu}{\nu}\kf_{\bflambdaq} .
\end{eqnarray*}
This proves that $A\cap B$ contains the unique point $(\bflambdaq,\hat\bfh)$.
\epr

\subsection{Proof of Proposition~\ref{prop:approx-sol}}
The strong convexity of $\log J(\cdot)$ in \eqref{INEQ:logJ-Hes-LB} implies that
\begin{align*}
\|\hat{\bfpsi}-\bfpsi_{X}(\omega^2,\nu)\|_2^2
\leq \frac{2}{A_1}\left(\log J(\hat{\bfpsi})-\log J(\bfpsi_{X}(\omega^2,\nu))\right)
\leq 2 \epsilon /A_1 .
\end{align*}
Now plug the above inequality into the following equation
\begin{align*}
\|\hat{\bfpsi}-\bfeta\|_2^2
=\|\bfpsi_{X}(\omega^2,\nu)-\bfeta\|_2^2 + 2     \|\hat{\bfpsi}-\bfpsi_{X}(\omega^2,\nu)\|_2
\|\bfpsi_{X}(\omega^2,\nu)-\bfeta\|_2
+\|\hat{\bfpsi}-\bfpsi_{X}(\omega^2,\nu)\|_2^2 ,
\end{align*}
we obtain the desired bound.
\epr

\subsection{Proof of Proposition~\ref{PROP:GMA-approxlogJ}}
From definition, $\bfpsi_{X}(\omega^2,\nu)=\kf_{\bflambda}$ with
$\bflambda\in\Lambda^M$ defined as
\begin{eqnarray*}
\lambda_j \propto \pi_j \exp\left(-\frac{1}{2\omega^2}\|\bff_j-\bfY\|_2^2+\frac{1-\nu}{2\omega^2}\|\bfpsi_{X}(\omega^2,\nu)-\bff_j\|_2^2\right) .
\end{eqnarray*}

For any $j=1,\ldots,M$,
\begin{align*}
\log J (\bfpsi^{(k)})
&=\log J \left(\bfpsi^{(k-1)} + \alpha_k(\bff_{J^{(k)}} - \bfpsi^{(k-1)})\right) \\
&\leq \log J \left(\bfpsi^{(k-1)} + \alpha_k(\bff_j - \bfpsi^{(k-1)})\right) \\
&\leq \log J (\bfpsi^{(k-1)})+ \alpha_k(\bff_j - \bfpsi^{(k-1)})^\top \frac{\nabla J(\bfpsi^{(k-1)})}{J(\bfpsi^{(k-1)})} + 2 \alpha_k^2 A_3,
\end{align*}
where the first inequality comes from definition, the second inequality is from Taylor expansion at $\bfpsi^{(k-1)}$ and \eqref{INEQ:logJ-Hes-UB}
in Lemma~\ref{LEM:logJ-2ord-convexity} with the fact that $\|\bff_j - \bfpsi^{(k-1)}\|_2^2\leq 4L^2$.

We multiply the above inequality by $\lambda_j$ and sum over $j$ to obtain
\begin{align*}
\log J (\bfpsi^{(k)})
&\leq \log J (\bfpsi^{(k-1)})+ \alpha_k \sum_{j=1}^M \lambda_j(\bff_j - \bfpsi^{(k-1)})^\top \frac{\nabla J(\bfpsi^{(k-1)})}{J(\bfpsi^{(k-1)})} + 2\alpha_k^2 A_3 \\
&= \log J (\bfpsi^{(k-1)})+ \alpha_k (\bfpsi_{X}(\omega^2,\nu) - \bfpsi^{(k-1)})^\top \frac{\nabla J(\bfpsi^{(k-1)})}{J(\bfpsi^{(k-1)})} + 2\alpha_k^2 A_3 \\
&\leq \log J (\bfpsi^{(k-1)})+ \alpha_k (\log J(\bfpsi_{X}(\omega^2,\nu))-\log J(\bfpsi^{(k-1)})) + 2\alpha_k^2 A_3 ,
\end{align*}
where the last inequality follows from the convexity of $\log J(\bfpsi)$.

Denote by $\delta_k = \log J (\bfpsi^{(k)}) -\log J(\bfpsi_{X}(\omega^2,\nu))$, it follows that
\begin{eqnarray*}
\delta_k \leq (1-\alpha_k) \delta_{k-1}  + 2\alpha_k^2 A_3 .
\end{eqnarray*}

We now bound $\delta_0$. If we let $\mu_j \propto \pi_j \exp\left(-\frac{1}{2\omega^2}\|\bff_j-\bfY\|_2^2\right)$ such that $\sum_{j=1}^M \mu_j =1$,
then
\begin{align}
\delta_0
&= \log J (\bfpsi^{(0)}) -\log J(\bfpsi_{X}(\omega^2,\nu)) 
\nonumber
\\
&= \log \sum_j \mu_j \exp\left(\frac{1-\nu}{2\omega^2}\|\bfpsi^{(0)}-\bff_j\|_2^2\right)
-\log \sum_j \mu_j \exp\left(\frac{1-\nu}{2\omega^2}\|\bfpsi_{X}(\omega^2,\nu)-\bff_j\|_2^2\right) 
\nonumber
\\
&\leq \log \left( \sum_{j=1}^M \mu_j \exp\left(\frac{1-\nu}{2\omega^2}\|\bfpsi^{(0)}-\bff_j\|_2^2\right) \right) \nonumber\\
&\leq  \frac{1-\nu}{2\omega^2} L^2 \leq 2A_3 . 
\label{eq:delta0}
\end{align}

The claim thus hold for $\delta_0$. By mathematical induction, if $\delta_{k-1} \leq \frac{8A_3}{k+2}$, then
\begin{align*}
\delta_k
\leq (1-\alpha_k) \delta_{k-1}  + 2\alpha_k^2 A_3
\leq (1-2/(k+1)) \frac{8A_3}{k+2} + 2(2/(k+1))^2 A_3 \leq \frac{8A_3}{k+3} .
\end{align*}
This proves the desired bound.
\epr

\subsection{Proof of Proposition~\ref{prop:opt-GMA}}
We have
\begin{align*}
\|\bfpsi^{(k)}-\bfpsi_{X}(\omega^2,\nu)\|_2^2
\leq \frac{2}{A_1}\left(\log J(\bfpsi^{(k)})-\log J(\bfpsi_{X}(\omega^2,\nu))\right)
\leq \frac{2}{A_1}\frac{8A_3}{k+3} = \frac{16A_3}{A_1(k+3)},
\end{align*}
where the first inequality comes from Taylor expansion at point $\bfpsi_{X}(\omega^2,\nu)$, with using \eqref{INEQ:logJ-Hes-LB} in Lemma~\ref{LEM:logJ-2ord-convexity} and $\nabla J(\bfpsi_2)$; and the second inequality is from Proposition~\ref{PROP:GMA-approxlogJ}.
It follows that
\begin{align*}
\|\bfpsi^{(k)}-\bfeta\|_2^2
&=\|(\bfpsi^{(k)}-\bfpsi_{X}(\omega^2,\nu))+(\bfpsi_{X}(\omega^2,\nu)-\bfeta)\|_2^2 
\\
&\leq\|\bfpsi_{X}(\omega^2,\nu)-\bfeta\|_2^2 
+ 2  \|\bfpsi_{X}(\omega^2,\nu)-\bfeta\|_2 \|\bfpsi^{(k)}-\bfpsi_{X}(\omega^2,\nu)\|_2 
+ \|\bfpsi^{(k)}-\bfpsi_{X}(\omega^2,\nu)\|_2^2 
\\
&\leq\|\bfpsi_{X}(\omega^2,\nu)-\bfeta\|_2^2
+ 2\sqrt{\frac{16A_3}{A_1(k+3)}}\|\bfpsi_{X}(\omega^2,\nu)-\bfeta\|_2 
+ \frac{16A_3}{A_1(k+3)}.
\end{align*}
\epr

\subsection{Proof of Proposition~\ref{PROP:GD-approxlogJ}}
The update operation implies
\[
\bfpsi^{(k)}= \bfpsi^{(k-1)}-t_k \nabla\log J(\bfpsi^{(k-1)}).
\]

Then, 
\begin{align*}
\log J(\bfpsi^{(k)})
&=\log J(\bfpsi^{(k-1)}-t_k \nabla\log J(\bfpsi^{(k-1)}))\\
&\leq\log J(\bfpsi^{(k-1)}) - t_k \|\nabla\log J(\bfpsi^{(k-1)}))\|_2^2 
+ (A_2/2) t_k^2 \|\nabla\log J(\bfpsi^{(k-1)}))\|_2^2 
\\
&=\log J(\bfpsi^{(k-1)}) - (t_k -(A_2/2)t_k^2)\|\nabla\log J(\bfpsi^{(k-1)}))\|_2^2,
\end{align*}
where the inequality is from \eqref{INEQ:logJ-Hes-UB}.
By subtracting $\log J(\bfpsi_{X}(\omega^2,\nu))$ by each side, we have
\begin{align}
\log J(\bfpsi^{(k)}) - \log J(\bfpsi_{X}(\omega^2,\nu))
\leq\ &\log J(\bfpsi^{(k-1)})- \log J(\bfpsi_{X}(\omega^2,\nu)) 
\nonumber
\\
&- (t_k -(A_2/2)t_k^2)\|\nabla\log J(\bfpsi^{(k-1)}))\|_2^2.
\label{INEQ:PROP-GD-approxlogJ-1}
\end{align}

Also from \eqref{INEQ:logJ-Hes-UB} we have
\begin{align}
\|\nabla\log J(\bfpsi^{(k-1)}))\|_2^2 
\geq2A_1 \left(\log J(\bfpsi^{(k-1)})- \log J(\bfpsi_{X}(\omega^2,\nu))\right).
\label{INEQ:PROP-GD-approxlogJ-2}
\end{align}
Choose fixed step size $t_k=s \in(0,2/A_2)$ for any $k>0$. Combining \eqref{INEQ:PROP-GD-approxlogJ-1} and \eqref{INEQ:PROP-GD-approxlogJ-2} results
\begin{align*}
\log J(\bfpsi^{(k)}) - \log J(\bfpsi_{X}(\omega^2,\nu)) 
\leq[1- 2A_1(s -(A_2/2)s^2)]  \left(\log J(\bfpsi^{(k-1)})- \log J(\bfpsi_{X}(\omega^2,\nu))\right).
\end{align*}
It follows that
\begin{align*}
\log J(\bfpsi^{(k)}) - \log J(\bfpsi_{X}(\omega^2,\nu)) 
\leq[1- 2A_1(s -(A_2/2)s^2)]^k  \left(\log J(\bfpsi^{(0)})- \log J(\bfpsi_{X}(\omega^2,\nu))\right).
\end{align*}
\epr

\subsection{Proof of Proposition~\ref{PROP:opt-GD}}
We choose $t_k=s$ as in Remark \ref{rema:1}. Then,
\begin{align*}
\|\bfpsi^{(k)}-\bfpsi_{X}(\omega^2,\nu)\|_2^2
&\leq\frac{2}{A_1}\left(\log J(\bfpsi^{(k)})-\log J(\bfpsi_{X}(\omega^2,\nu))\right) 
\\
&\leq\frac{2}{A_1}(1-A_1/A_2)^k \log J(\bfpsi^{(0)}) 
\\
&\leq\frac{2}{A_1}(1-A_1/A_2)^k \frac{1-\nu}{2\omega^2}L^2 
\\
&= L^2(1-A_1/A_2)^k,
\end{align*}
where the first inequality comes from Taylor expansion at point $\bfpsi_{X}(\omega^2,\nu)$, with using \eqref{INEQ:logJ-Hes-LB} in Lemma~\ref{LEM:logJ-2ord-convexity} and $\nabla \log J(\bfpsi_{X}(\omega^2,\nu))=0$;
the second inequality is from Proposition~\ref{PROP:GD-approxlogJ};
and the third inequality is from assumption \eqref{CON:f-l2normbound} resulting $\log J(\bfpsi^{(0)}) \leq \frac{1-\nu}{2\omega^2}L^2$.
It follows that
\begin{align*}
\|\bfpsi^{(k)}-\bfeta\|_2^2
&=\|(\bfpsi^{(k)}-\bfpsi_{X}(\omega^2,\nu))+(\bfpsi_{X}(\omega^2,\nu)-\bfeta)\|_2^2 \\
&\leq\|\bfpsi_{X}(\omega^2,\nu)-\bfeta\|_2^2
+ 2  \|\bfpsi_{X}(\omega^2,\nu)-\bfeta\|_2 \|\bfpsi^{(k)}-\bfpsi_{X}(\omega^2,\nu)\|_2 
+\|\bfpsi^{(k)}-\bfpsi_{X}(\omega^2,\nu)\|_2^2
\\
&\leq\|\bfpsi_{X}(\omega^2,\nu)-\bfeta\|_2^2
+ 2\sqrt{L^2(1-A_1/A_2)^k}\|\bfpsi_{X}(\omega^2,\nu)-\bfeta\|_2 
+L^2(1-A_1/A_2)^k.
\end{align*}
\epr

\subsection{Proof of Proposition~\ref{PROP:GD-approxlogJ-MH}}
Given $\bfY$, the expectation is with respect to the randomness from the MH algorithm. For $k>0$, $\bfu_T^{(k-1)}$ from Algorithm~\ref{ALG:MH} is an estimator of $\kf_{\bflambda^{(k-1)}}= \sum_{j=1}^M \lambda_j^{(k-1)} \bff_j$. Then in Algorithm~\ref{ALG:GD}, we update $\bfpsi^{(k)}$ by
\[
\bfpsi^{(k)}= \bfpsi^{(k-1)} - t_k \frac{1-\nu}{\omega^2} (\bfpsi^{(k-1)}-\bfu_T^{(k-1)}).
\]
Denote $v^{(k-1)} = \frac{1-\nu}{\omega^2} (\bfpsi^{(k-1)}-\bfu_T^{(k-1)})$, then we have
\begin{eqnarray*}
\E[v^{(k-1)}|\bfpsi^{(k-1)}] = \frac{1-\nu}{\omega^2} (\bfpsi^{(k-1)}-\kf_{\bflambda^{(k-1)}}) = \nabla\log J(\bfpsi^{(k-1)}),
\end{eqnarray*}
and
\begin{align*}
\|COV[v^{(k-1)}|\bfpsi^{(k-1)}] \|_{op} 
=\left(\frac{1-\nu}{\omega^2}\right)^2 \|COV[\bfu_T^{(k-1)}|\bfpsi^{(k-1)}] \|_{op}
\leq \left(\frac{1-\nu}{\omega^2}\right)^2 s^2.
\end{align*}
It follows that
\begin{align*}
\log J(\bfpsi^{(k)})
=\ &\log J(\bfpsi^{(k-1)}-t_k v^{(k-1)})\\
\leq\ &\log J(\bfpsi^{(k-1)}) - t_k \nabla\log J(\bfpsi^{(k-1)})^\top v^{(k-1)} 
+ (A_2/2) t_k^2 \|v^{(k-1)}\|_2^2,
\end{align*}
where the inequality is from \eqref{INEQ:logJ-Hes-UB}.
Then, by subtracting $\log J(\bfpsi_{X}(\omega^2,\nu))$ from each side of the above equation and taking expectation conditioned on $\bfpsi^{(k-1)}$, we have
\begin{align*}
\E [\delta_k | \bfpsi^{(k-1)}]
&\leq\delta_{k-1} - t_k \nabla\log J(\bfpsi^{(k-1)})^\top \E[v^{(k-1)}|\bfpsi^{(k-1)}] 
+ (A_2/2) t_k^2 \E[\|v^{(k-1)}\|_2^2|\bfpsi^{(k-1)}] 
\\
&\leq\delta_{k-1} - t_k\|\nabla\log J(\bfpsi^{(k-1)})\|_2^2
+(A_2/2) t_k^2 \left(\|\nabla\log J(\bfpsi^{(k-1)})\|_2^2+ n\left(\frac{1-\nu}{\omega^2}\right)^2s^2 \right) 
\\
&=\delta_{k-1} - \frac{1}{2A_2}\|\nabla\log J(\bfpsi^{(k-1)})\|_2^2 + \frac{1}{2A_2}\left(\frac{1-\nu}{\omega^2}\right)^2 ns^2,
\end{align*}
where $\delta_k= \log J(\bfpsi^{(k)})- \log J(\bfpsi_{X}(\omega^2,\nu)) $ and $t_k=s=1/A_2$ as in Remark \ref{rema:1}. Combining the above inequality with
\begin{align}
\|\nabla\log J(\bfpsi^{(k-1)}))\|_2^2 \geq 2A_1 \Big(&\log J(\bfpsi^{(k-1)})
-\log J(\bfpsi_{X}(\omega^2,\nu))\Big),
\end{align}
which is from \eqref{INEQ:logJ-Hes-UB}, we have
\begin{eqnarray*}
\E [\delta_k | \bfpsi^{(k-1)}]
\leq \delta_{k-1} (1-A_1/A_2) + \frac{A_1^2}{2A_2} ns^2.
\end{eqnarray*}
Therefore, it follows that
\begin{eqnarray*}
\E [\delta_k]
\leq \E[\delta_{k-1}] (1-A_1/A_2) + \frac{A_1^2}{2A_2} ns^2,
\end{eqnarray*}
and we have
\begin{eqnarray*}
\E [\delta_k]
\leq \E[\delta_{0}] (1-A_1/A_2)^k + \frac{A_1}{2} ns^2.
\end{eqnarray*}
\epr

\bibliographystyle{unsrt}
\bibliography{ref}

\end{document}